\documentclass[12pt,twoside,reqno]{amsart}

\usepackage{amssymb}
\usepackage{amsmath}
\usepackage{mathrsfs}
\usepackage{graphicx}

\def\XX{{\mathcal{X}}}
\def\ZZ{{\mathcal{Z}}}
\def\UU{{\mathcal{U}}}
\def\R{{\mathbb R}}
\def\C{{\mathbb{C}}}
\def\Mh{{\mathfrak{M}_h}}

\def\d{\delta}
\def\D{\Delta}
\def\e{\varepsilon}
\def\f{\varphi}
\def\s{\sigma}
\def\a{\alpha}
\def\g{\gamma}
\def\l{\lambda}
\def\L{\Lambda}
\def\k{\varkappa}

\newtheorem{theorem}{Theorem}[section]
\newtheorem{lemma}[theorem]{Lemma}
\newtheorem{corollary}[theorem]{Corollary}

\newtheorem{remark}[theorem]{Remark}
\newtheorem{example}[theorem]{Example}

\numberwithin{equation}{section}
\newcommand{\supp}{\mathop{\rm supp}\nolimits}
\newcommand{\Var}{\mathop{\rm Var}\nolimits}

\begin{document}

\title[Vector-valued Riesz potentials]
{Vector-valued Riesz potentials:\\
Cartan type estimates and related capacities}

\author{V.~Eiderman}
\address{Vladimir Eiderman, Department of  Mathematics, University of Wisconsin-Madison, Madison, WI}
\email{eiderman@ms.uky.edu}
\author{F.~Nazarov}
\address{Fedor Nazarov, Department of  Mathematics, University of Wisconsin-Madison, Madison, WI}
\email{nazarov@math.wisc.edu}
\author{A.~Volberg}
\address{Alexander Volberg, Department of  Mathematics, Michigan State University and the University of Edinburgh}
\email{volberg@math.msu.edu, a.volberg@ed.ac.uk}
\thanks{Research of the authors was supported in part by NSF grants  DMS-0501067 (Nazarov, Volberg).\\
The first-named author is grateful to the Department of  Mathematics, Michigan State University, for its hospitality.}
\begin{abstract}
Our aim is to give sharp upper bounds for the size of the set of points where the Riesz transform
of a linear combination of $N$ point masses is large. This size will be measured by the Hausdorff content with various
gauge functions. Among other things, we shall characterize all gauge functions for which the estimates do not blow up as
$N$ tends to infinity (in this case a routine limiting argument will allow us to extend our bounds to all finite
Borel measures). We also show how our techniques can be applied to estimates for certain capacities.
\end{abstract}

\maketitle

\section{Introduction}

The aim of this paper is to give sharp upper bounds for the size of the set of points where the singular $s$-Riesz transform of a
linear combination of $N$ point masses is large. This size will be measured by the Hausdorff content $M_h$ with various
gauge functions $h$. Among other things, we shall characterize all $h$ for which the estimates do not blow up as
$N\to+\infty$ (in this case a routine limiting argument will allow us to extend our bounds to all finite Borel measures).

The main motivation for this work came from \cite{AEi}, \cite{E1} where such bounds were obtained for the Cauchy transform
in $\C$, which is a special case of $s$-Riesz transform in $\R^d$ corresponding to $d=2$, $s=1$.

The main challenge in the generalizing the result to higher dimensions and different $s$ is the absence of the Menger's
curvature tool. The methods of the current paper allow us to cover the range $0<s<\infty$.

\section{Definitions and notation}

Let $\nu$ be a finite Borel measure in $\mathbb R^d$ (not necessarily positive). For $s>0$, put
$$
K^s(x)=\frac x{|x|^{s+1}},\quad x\in\mathbb R^d\setminus\{0\}.
$$
For $\e>0$, define the $\e$-truncated $s$-Riesz transform of $\nu$ by
$$
R_{\nu,\e}^s(x)=\int_{|y-x|>\e}K^s(y-x)\,d\nu(y).
$$
If the limit $\lim_{\e\to0+}R_{\nu,\e}^s(x)$ exists, we shall call it the $s$-Riesz transform of $\nu$ at $x$
and denote it by $R_{\nu}^s(x)$. For $0<s<d$, this is true almost everywhere in $\R^d$ with respect to Lebesgue measure
but this observation will be rather useless to us because sets of Lebesgue measure 0 can
easily have arbitrarily large Hausdorff content $M_h$ when $\lim_{t\to0+}\frac{h(t)}{t^{d}}=+\infty$ (we remind the
definition of $M_h$ below). Of course, in the case when $\nu$ is a finite linear combination of point masses,
$R_{\nu}^s(x)$ makes sense everywhere except finitely many points. Nevertheless, since we are aiming at
extending our results to all finite Borel measures whenever possible, we shall introduce one more quantity
that always makes sense, namely the so called maximal $s$-Riesz transform
$$
R_{\nu,\ast}^s(x)=\sup_{\e>0}|R_{\nu,\e}^s(x)|
$$
(note that $R_{\nu,\e}^s(x)$ and $R_{\nu}^s(x)$ are vectors and $R_{\nu,\ast}^s(x)$ is a number).

Take $P>0$. The sets whose size we will be interested in are
$$
\ZZ(\nu,P)=\{x\in \R^d: R_{\nu}^s(x) \text{ exists and }|R_{\nu}^s(x)|>P\}
$$
and
$$
\ZZ^\ast(\nu,P)=\{x\in \R^d: R_{\nu,\ast}^s(x)>P\}.
$$
Clearly, $\ZZ(\nu,P)\subset\ZZ^\ast(\nu,P)$.

By a {\it gauge} (or {\it measuring}) function, we shall understand any continuous strictly increasing function
$h:[0,+\infty)\to[0,+\infty)$ such that $h(0)=0$, $\lim_{r\to+\infty}h(r)=+\infty$ and $\frac{h(r)}{r^d}$
is non-increasing. The last condition, which may seem a regularity condition at the first glance, is actually
not a restriction at all. As we shall see, for any increasing function $h$ vanishing at the origin,
there exists another function $\widetilde h$ satisfying this restriction and such that the Hausdorff contents
$M_h$ and $M_{\tilde h}$ coincide up to a constant factor depending only on the dimension $d$.

The {\it Hausdorff content $M_h(G)$} of a set $G\subset\R^d$ is defined by
$$
M_h(G)=\inf\sum_j h(r_j),
$$
where the infimum is taken over all (at most countable) coverings of $G$ by balls of radii $r_j$.

By $c,\ C$ we denote various positive constants, and we set $B(x,r):=\{y\in\R^d:\ |y-x|<r\}$.
To show that the assumption that $\frac{h(r)}{r^d}$ is non-increasing is not restrictive, we consider a
continuous increasing function $h$ satisfying $h(0)=0$, without this property in general. Set
$$
\widetilde h(r)=r^d\inf_{0<t\le r}\frac{h(t)}{t^d}.
$$
Clearly, $\widetilde h(r)\le h(r)$, and $\frac{\widetilde h(r)}{r^d}$ is non-increasing. We are going to prove that
$$
M_{\widetilde h}(G)\le M_{h}(G)\le C_dM_{\widetilde h}(G)
$$
for any set $G$, with $C_d$ depending only on $d$. The first inequality is obvious. Choose $\e>0$. Let $\{B_j\}$,
$B_j=B(x_j,r_j)$, be a covering of $G$ by balls. For every ball $B_j$, we take $t_j$ such that $0<t_j\le r_j$ and
$$
\frac{\e}{2^j}+\widetilde h(r_j)>r_j^d\,\frac{h(t_j)}{t_j^d}.
$$
There is a constant $C_d$ depending only on $d$ such that $B_j$ can be covered by at most $C_d(\frac{r_j}{t_j})^d$
balls $B_{j,k}$ of radii $r_{j,k}=t_j$. For this new covering of $G$, we have
$$
\sum_{j,k}h(r_{j,k})=\sum_jC_d\bigg(\frac{r_j}{t_j}\bigg)^dh(t_j)<C_d\sum_j\widetilde h(r_j)+C_d\e.
$$
Thus, $M_{h}(G)\le C_d M_{\widetilde h}(G)$, and we are done.

\medskip

For nonnegative Borel measure $\mu$, we introduce the $\e$-truncated $s$-Riesz operator defined by
$$
\mathfrak{R}_{\mu,\e}^s f(x)=\int_{|y-x|>\e}K^s(y-x)f(y)\,d\mu(y),\quad f\in L^2(\mu),\quad \e>0.
$$
For every $\e>0$, the operator $\mathfrak{R}_{\mu,\e}^s$ is bounded on $L^2(\mu)$. We set
$$
\pmb|\mathfrak{R}_{\mu}^s\pmb|:=\sup_{\e>0}\|\mathfrak{R}_{\mu,\e}^s\|_{L^2(\mu)\to L^2(\mu)}.
$$

\section{Function $\Mh(\k,N)$}

Let $h$ be a measuring function, $\k>0,\ N\ge2$, and let $h^{-1}$ be inverse to $h$.
Define $\Mh(\k,N)$ as the unique solution $M>0$ of the equation
\begin{equation}\label{f31}
\k^2\int_{1/N}^1\bigg[\frac{t}{h^{-1}(Mt)^s}\bigg]^2\frac{dt}{t}=1.
\end{equation}
Note that the left hand side is a decreasing function of $M$ that tends to $+\infty$ at 0 and to 0 at $+\infty$,
so the definition makes sense. Moreover, it is clear that $\Mh$ is increasing in both $\k$ and $N$.

\begin{lemma}\label{le31} (Doubling property of $\Mh$). For each measuring function $h$ and $0<s<d$,
$$
\Mh(2\k,2N)\le C(s,d)\,\Mh(\k,N).
$$
\end{lemma}

\begin{proof} It is enough to show that for sufficiently large $C$, one has
$$
4\int_{1/(2N)}^1\bigg[\frac{t}{h^{-1}(CMt)^s}\bigg]^2\frac{dt}{t}\le
\int_{1/N}^1\bigg[\frac{t}{h^{-1}(Mt)^s}\bigg]^2\frac{dt}{t}.
$$
Assuming that $C>2$, and making the change of variable $t\mapsto Ct$, we see that the left hand side does not exceed
$$
\frac4{C^2}\int_{1/N}^C\bigg[\frac{t}{h^{-1}(Mt)^s}\bigg]^2\frac{dt}{t}\,.
$$
Now notice that the function
$$
H(t)=\frac{t^{s/d}}{h^{-1}(Mt)^s}
$$
is non-increasing. Therefore,
$$
\int_1^C\left[t^{1-\frac sd}H(t)\right]^2\frac{dt}{t}\le
H(1)^2\int_1^Ct^{2(1-\frac sd)}\frac{dt}{t}\le
H(1)^2\frac1{2(1-\frac sd)}C^{2(1-\frac sd)}\,.
$$
Also,
$$
\int_{1/N}^1\left[t^{1-\frac sd}H(t)\right]^2\frac{dt}{t}\ge
H(1)^2\int_{1/2}^1t^{2(1-\frac sd)}\frac{dt}{t}\ge
H(1)^2\cdot\frac38\,.
$$
Hence,
$$
\frac4{C^2}\int_{1/N}^C\bigg[\frac{t}{h^{-1}(Mt)^s}\bigg]^2\frac{dt}{t}\le
\frac4{C^2}\bigg[1+\frac4{3(1-\frac sd)}C^{2(1-\frac sd)}\bigg]
\int_{1/N}^1\bigg[\frac{t}{h^{-1}(Mt)^s}\bigg]^2\frac{dt}{t}\,,
$$
and it remains to note that the factor in front of the integral on the right hand side tends to 0 as $C\to\infty$.
\end{proof}

\begin{lemma}\label{le32} Let $\k>0,\ N\ge 2,\ c\in(0,1)$, and $0<s<d$. If
$$
M\le C\k\left[\int_{h^{-1}(\frac cNM)}^{h^{-1}(M)}\bigg(\frac{h(y)}{y^s}\bigg)^2\frac{dy}{y}\right]^{1/2}\,,
$$
then $M\le C'(c,C,s,d)\,\Mh(\k,N)$.
\end{lemma}

\begin{proof}
Let $n$ be the least positive integer satisfying $2^{-n}<c/N$. Note that
\begin{multline*}
\int_{h^{-1}(\frac cNM)}^{h^{-1}(M)}\bigg[\frac{h(y)}{y^s}\bigg]^2\frac{dy}{y}
\le \int_{h^{-1}(2^{-n}M)}^{h^{-1}(M)}\bigg[\frac{h(y)}{y^s}\bigg]^2\frac{dy}{y}\\
=\sum_{j=1}^n\int_{h^{-1}(2^{-j}M)}^{h^{-1}(2^{-(j-1)}M)}\bigg[\frac{h(y)}{y^s}\bigg]^2\frac{dy}{y}
\le C(s)\sum_{j=1}^n\bigg[\frac{2^{-j}M}{h^{-1}(2^{-j}M)^s}\bigg]^2\\
\le C''(s)M^2\sum_{j=1}^n\int_{2^{-(j+1)}}^{2^{-j}}\bigg[\frac{t}{h^{-1}(tM)^s}\bigg]^2\frac{dt}{t}
\le C''(s)M^2\int_{c/(4N)}^{1}\bigg[\frac{t}{h^{-1}(Mt)^s}\bigg]^2\frac{dt}{t}.
\end{multline*}
Thus the inequality in the condition of the lemma implies that
$$
\sqrt{C''(s)}C\k\left[\int_{c/(4N)}^{1}\bigg(\frac{t}{h^{-1}(Mt)^s}\bigg)^2\frac{dt}{t}\right]^{1/2}\ge1,
$$
i.e.,
$$
M\le\Mh(\sqrt{C''(s)}C\k,\tfrac4cN)\le C'(c,C,s,d)\,\Mh(\k,N).
$$
\end{proof}

\begin{remark}\label{re33} Let $c,C>0,\ \k>0,\ N>2C$, and $0<s<d$. If
$$
M\ge c\k\left[\int_{h^{-1}(\frac CNM)}^{h^{-1}(M)}\bigg(\frac{h(y)}{y^s}\bigg)^2\frac{dy}{y}\right]^{1/2}\,,
$$
then $M\ge c'(c,C,s,d)\,\Mh(\k,N)$.
\end{remark}

\begin{proof} For $y=h^{-1}(Mt)$,

\begin{equation}\label{f32}
\begin{split}
\int_{h^{-1}(\frac CNM)}^{h^{-1}(M)}\bigg[\frac{h(y)}{y^s}\bigg]^2\frac{dy}{y}
&=M^2\int_{C/N}^1\bigg[\frac{t}{h^{-1}(Mt)^s}\bigg]^2\frac{dh^{-1}(Mt)}{h^{-1}(Mt)}\\
&\ge\frac{M^2}{d}\int_{C/N}^1\bigg[\frac{t}{h^{-1}(Mt)^s}\bigg]^2\frac{dt}t
\end{split}
\end{equation}
because, since $t\mapsto[\log h^{-1}(Mt)-\frac1d\log(Mt)]$ is non-decreasing, we have
$$
\frac{dh^{-1}(Mt)}{h^{-1}(Mt)}\ge\frac1d\,\frac{dt}t.
$$
Thus, the condition in the remark implies
$$
\frac1{\sqrt d}c\k\left[\int_{C/N}^1\bigg(\frac{t}{h^{-1}(Mt)^s}\bigg)^2\frac{dt}t\right]^{1/2}\le 1,
$$
i.e.,
$$
M\ge\Mh\bigg(\frac{c}{\sqrt{d}}\k,\frac NC\bigg)\ge c\,\Mh(\k,N).
$$
\end{proof}
Hence, the condition in Lemma \ref{le32} is essentially equivalent to $M\le C\,\Mh(\k,N)$.

\section{Main results}

We formulate our results for $0<s<d$ and for $s\ge d$ separately. The main and the most difficult case is $0<s<d$.

\begin{theorem}\label{th41}
Let $N\ge 2$, $P>0,\ s\in(0,d)$. Let $h$ be any measuring function.

I) There exists $C=C(s,d)>0$ such that for every measure $\nu$ that is a linear combination of $N$ Dirac point
masses, the inequality

\begin{equation}\label{f41}
M_h(\ZZ^\ast(\nu,P))\le C\,\Mh\bigg(\frac{\|\nu\|}{P},N\bigg)
\end{equation}
holds.

II) There exists $c=c(s,d)>0$ such that, for every $\eta>0$, one can find a measure $\nu$ that is a linear combination of $N$ Dirac point
masses and such that
$\|\nu\|=\eta$ and

\begin{equation}\label{f42}
M_h(\ZZ(\nu,P))\ge c\,\Mh\bigg(\frac{\|\nu\|}{P},N\bigg).
\end{equation}
\end{theorem}

Theorem \ref{th41} is a generalization of the corresponding results in \cite{AEi}, \cite{E1}, \cite{E2}.
In some important cases one can derive explicit estimates for $M_h(\ZZ^\ast(\nu,P))$ from (\ref{f41}).

\begin{example}\label{ex42}
For $h(t)=t^\beta$, $s<d$ and $\beta\le d$, easy calculations yield

\begin{equation}\label{f43}
M_h(\ZZ^\ast(\nu,P))\le\left\{\begin{array}{ll}
C\biggl(\dfrac{\|\nu\|}{P}\biggr)^{\beta/s}\biggl(\dfrac{N^{2(s-\beta)/\beta}-1}{s-\beta}\biggr)^{\beta/(2s)},&\beta\ne s,\\
C\dfrac{\|\nu\|}{P}(\ln N)^{1/2},&\beta=s,\ N\ge2,
\end{array}\right.
\end{equation}
where $C$ depends only on $d$ and $s$.
\end{example}

For $\beta>d$ the $h$-content of every set in $\R^d$ is zero.

It is interesting to compare inequalities (\ref{f43}) with estimates for the $h$-content of the set
$$
\XX(|\nu|,P):=\biggl\{x\in \R^d: \int|K^s(y-x)|\,d|\nu|(y)=\int\frac1{|y-x|^s}\,d|\nu|(y)>P\biggr\}.
$$
Obviously, $\ZZ^\ast(\nu,P)\subset\XX(|\nu|,P)$. Corollary 1.2 in \cite{E2} yields the following estimate: for $h(t)=t^\beta$ and $\beta\le d$,

\begin{equation}\label{f44}
M_h(\XX(|\nu|,P))\le\left\{\begin{array}{ll}
C\biggl(\dfrac{\|\nu\|}{P}\biggr)^{\beta/s}\biggl(\dfrac{N^{(s-\beta)/\beta}-1}{s-\beta}\biggr)^{\beta/s},&\beta\ne s,\\
C\dfrac{\|\nu\|}{P}\ln N,&\beta=s,\ N\ge2,
\end{array}\right.
\end{equation}
with another constant $C$ depending only on $d$ and $s$.

If $s<\beta$, then the estimate
$$
M_h(\XX(|\nu|,P))\le C\biggl(\frac{\|\nu\|}{P}\cdot\dfrac{1}{\beta-s}\biggr)^{\beta/s}
$$
(the limiting case of (\ref{f44}) as $N\to+\infty$) holds for every (not necessarily discrete) measure $\nu$.
The exponent 1/2 in (\ref{f43}) reflects the mutual annihilation of terms in the passage from the sum of moduli to the modulus of the sum of the corresponding fractions.

Consider now the case when

\begin{equation}\label{f45}
\int_0\bigg[\frac{h(t)}{t^{s}}\bigg]^2\frac{dt}t<\infty.
\end{equation}

Under this assumption we obtain estimates for the $h$-content of $M_h(\ZZ^\ast(\nu,P))$ not only for discrete
measures but also for arbitrary finite Borel measures $\nu$. (Note that for the function $h(t)=t^\beta$ with
$0<\beta\le d$, the condition (\ref{f45}) holds iff $s<\beta\le d$. This is exactly the case when the right-hand sides
of (\ref{f43}) and (\ref{f44}) do not blow up as $N\to+\infty$.)

The condition (\ref{f45}) implies that for every $M>0$,
$$
\int_0^1\bigg[\frac{t}{h^{-1}(Mt)^s}\bigg]^2\frac{dt}{t}<\infty
$$
(see (\ref{f32})). Since this integral is a decreasing function of $M$ that tends to $+\infty$ at 0 and to 0 at $+\infty$,
the equation
$$
\k^2\int_0^1\bigg[\frac{t}{h^{-1}(Mt)^s}\bigg]^2\frac{dt}{t}=1
$$
has the unique solution $M>0$, which we denote by $\Mh(\k,\infty)$.

\begin{theorem}\label{th43}
Let $\nu$ be a Borel measure (generally, complex-valued) with finite total variation, and let $h$ be a measuring function
satisfying (\ref{f45}). Then for any $P>0$,

\begin{equation}\label{f46}
M_h(\ZZ^\ast(\nu,P))\le C\,\Mh\bigg(\frac{\|\nu\|}{P},\infty\bigg),
\end{equation}
where $C$ depends only on $d$ and $s$.
\end{theorem}

Theorem \ref{th43} can be viewed as a limiting case of Theorem \ref{th41} as $N\to\infty$.

\medskip

Let us consider the case $s\ge d$.

\begin{theorem}\label{th44}
Let $N \ge 2$, $P>0,\ s\ge d$, and let $h$ be any measuring function.
There exists $C=C(s,d)>0$ such that for every measure $\nu$ that is a linear combination of $N$ Dirac point
masses, the inequality

\begin{equation}\label{f48}
M_h(\ZZ^\ast(\nu,P))\le CNh\bigg(\bigg(\frac{\|\nu\|}{PN}\bigg)^{1/s}\bigg)
\end{equation}
holds.
\end{theorem}

Thus, for $s>d$ cancelation plays no role and the sharp estimate for $M_h(\ZZ^\ast(\nu,P))$ is the same (up to
a constant factor depending on $s$ and $d$) as the estimate for the $h$-Hausdorff content of the set $\XX(|\nu|,P)$
obtained in \cite{E2}, which in this case reduces to \eqref{f48}. Note that the right hand side of this estimate always
(when $s>d$) blows up as $N\to\infty$ (since $h(t)\ge c\,t^d,\ t\rightarrow 0+$).
So, no meaningful generalization to the case of arbitrary Borel measures $\nu$ is possible here.

The sharpness of the estimate \eqref{f48} for $s\ge d$ is obvious: just consider $N$ equal point masses located far away
from each other.

\medskip

Proofs of Theorems \ref{th41}, \ref{th43}  are based on the weak type 1 -- 1 estimate
for the maximal Calder\'on-Zygmund operator obtained by Nazarov, Treil and Volberg in \cite{NTV}, p.~483.
We quote this general result for the case of the maximal Riesz transform $R_{\nu,\ast}^s(x)$.

Let $\Sigma_s$ be the class of nonnegative Borel measures $\mu$ in $\R^d$ such that

\begin{equation}\label{f49}
\mu(B(x,r))\le r^s \quad\text{for all}\ x\in\R^d\text{ and }r>0.
\end{equation}

\begin{theorem}\label{th45} \cite{NTV} Suppose that $\mu\in\Sigma_s$ and $\pmb|\mathfrak{R}_{\mu}^s\pmb|<\infty$.
Then for every complex-valued Radon measure $\nu$ one has

\begin{equation}\label{f410}
\mu\{x\in\R^d:R_{\nu,\ast}^s(x)>t\}<\frac{C\|\nu\|}t
\end{equation}
with $C$ depending only on $s$ and $\pmb|\mathfrak{R}_{\mu}^s\pmb|$.
\end{theorem}

To apply Theorem \ref{th45} we should be able to construct the auxiliary measure $\mu$ and to estimate
$\pmb|\mathfrak{R}_{\mu}^s\pmb|$. This
estimate has various applications (see for example Section~11 below) and as we believe is of independent interest.

\begin{theorem}\label{th46}
For every nonnegative Borel measure $\mu$ and $0<s<d$ we have

\begin{equation}\label{f411}
\pmb|\mathfrak{R}_{\mu}^s\pmb|^2\le C\sup_{x\in\supp \mu}\int_0^\infty\bigg[\frac{\mu(B(x,r))}{r^s}\bigg]^2\,\frac{dr}{r}
\end{equation}
with $C$ depending only on $s$ and $d$.
\end{theorem}

As a byproduct of our calculations we obtain the following estimate for Calder\'on-Zygmund capacity $\g_{s,+}(E)$
(the corresponding definitions are given in Section~11).

\begin{theorem}\label{th47}
For any compact set $E\subset\R^d$,

\begin{equation}\label{f411'}
\g_{s,+}(E)\ge c\,\sup\|\mu\|^{3/2}\bigg[\int_{\R^d}W^\mu(x)\,d\mu(x)\bigg]^{-1/2},\quad
W^\mu(x):=\int_0^\infty\bigg[\frac{\mu(B(x,r))}{r^s}\bigg]^2\,\frac{dr}{r},
\end{equation}
where the supremum is taken over all positive Radon measures supported by $E$, and $c$ depends only on $d$, $s$.
\end{theorem}

We give two applications of Theorem \ref{th47}. We use it to derive relations between Hausdorff content
$M_h(E)$ and capacity $\g_{s,+}(E)$. Interestingly enough, there exists another completely different capacity
$C_{\frac23(d-s),\frac32}(E)$ (see \cite{AH}, \cite{MH}, \cite{HW} and the multiple references therein), that can be
characterized via the same potential $W^\mu(x)$. Theorem \ref{th47} immediately implies that
$\g_{s,+}(E)\ge c\cdot C_{\frac23(d-s),\frac32}(E)$ (see Section~11).

\vspace{.1in}

We prove Theorem \ref{th46} in Section 5. The construction of the appropriate measure $\mu$ is given in Section 6.
The first part of Theorem \ref{th41} is proved in Section~7 and the second part in Section~8. Section~9 contains the
proof of Theorem \ref{th43}. In Section~10 we consider the case $s\ge d$ and prove Theorem \ref{th44}.
In Section~11 we investigate metric properties of various capacities generated by vector-valued Riesz
potentials. In particular, we obtain the Frostman type theorem on comparison of these capacities and Hausdorff content.

\section{Proof of Theorem \ref{th46}}

The main trick in \cite{M} (which led to the use of Menger's curvature in non-homogeneous harmonic analysis)
is to symmetrize an expression involving Cauchy kernel by using averaging over all permutations of coordinates.
Amazingly this averaging a) is non-negative, b) is ``considerably smaller" than the absolute value of the
original expression, and c) is equal to a certain curvature. This observation is no longer true when one averages
a similar expression involving  vector Riesz kernels in $\R^d,\ d>2$, see the paper of Hany Farag \cite{F}.
This is why  we said (repeating the expression of Guy David) that the tool of curvature is ``cruelly missing" for $d>2$.

However, the following simple observation still  holds for all dimensions. If we symmetrize the pertinent expression
involving Riesz kernels we  generally miss a) and c) above, but we still have b): the symmetrized expression has
``considerably smaller" absolute value than the original one. This should be understood maybe not pointwise,
but in average.

This observation saves our day, proves Theorem \ref{th46}, and in general allows us to obtain very sharp estimates of various Calder\'on-Zygmund capasities $\gamma_{s, +}$ from below.

\begin{lemma}\label{le51}
Let $x,y,z$ be three distinct points in $\R^d$, $d\ge 1$, and let $|z-x|\le|z-y|\le|y-x|$. Then for $s>0$,

\begin{equation}\label{f51}
\aligned
q_s(x,y,z)&:=\frac{x-z}{|x-z|^{s+1}}\cdot\frac{y-z}{|y-z|^{s+1}}
+\frac{y-x}{|y-x|^{s+1}}\cdot\frac{z-x}{|z-x|^{s+1}}\\
&\le\frac{2^{s+1}}{|y-x|^{s+1}}\cdot\frac{1}{|z-x|^{s-1}}\,.
\endaligned
\end{equation}
\end{lemma}

\begin{proof}
Let $a=|y-x|$, $b=|z-y|$, $c=|z-x|$, and let $\a,\beta,\g$ be the angles opposite to sides $a,b,c$ respectively. Then
$$
q_s(x,y,z)=(abc)^{-s}(a^s\cos\a+b^s\cos\beta).
$$
Since
$$
\cos\alpha=\frac{b^2+c^2-a^2}{2bc},\quad\cos\beta=\frac{a^2+c^2-b^2}{2ac},
$$
we have with $u=b/a,\ v=c/a$
$$\aligned
q_s(x,y,z)&=2^{-1}(abc)^{-s-1}\left[a^{s+1}(b^2+c^2-a^2)+b^{s+1}(a^2+c^2-b^2)\right]\\
&=2^{-1}(uv)^{-s-1}a^{-2s}\left[u^2+v^2-1+u^{s+1}(1+v^2-u^2)\right]\\
&=2^{-1}(uv)^{-s-1}a^{-2s}\left[u^{s+1}(1-u^2)+v^2(1+u^{s+1})-(1-u^2)\right]\\
&=2^{-1}(uv)^{-s-1}a^{-2s}\left[v^2(1+u^{s+1})-(1-u^2)(1-u^{s+1})\right]\\
&\le u^{-s-1}v^{-s+1}a^{-2s}\le2^{s+1}a^{-s-1}c^{-s+1},
\endaligned$$
because $1/2\le u\le1$.
\end{proof}

\begin{proof}[Proof of Theorem \ref{th46}]
Without loss of generality we assume that

\begin{equation}\label{f52}
\mathbf S:=\sup_{x\in\supp\mu}\int_0^\infty\bigg[\frac{\mu(B(x,r))}{r^s}\bigg]^2\,\frac{dr}{r}<\infty.
\end{equation}
Otherwise (\ref{f411}) becomes trivial. We consider the measure
$$
\eta:=(2s\mathbf S)^{-1/2}\mu.
$$
Since for every $x\in\supp\mu$ and $r>0$,
$$
\mathbf S\ge\int_0^\infty\bigg[\frac{\mu(B(x,t))}{t^{s}}\bigg]^2\,\frac{dt}{t}\ge
[\mu(B(x,r))]^2\int_r^\infty\frac{dt}{t^{2s+1}}=\frac{[\mu(B(x,t))]^2}{2sr^{2s}},
$$
we see that $\eta\in\Sigma_s$.

We are going to give {\it two proofs} of the theorem. They have a lot in common but they are different. The first
one will be based on the non-homogeneous $T1$ theorem.

\noindent{\bf The first approach.} Let $Q$ be any cube in $\R^d$. If for each $\e>0$ we prove the inequality

\begin{equation}\label{f53}
\|\mathfrak{R}_{\eta,\e}^s\chi_Q\|^2_{L^2(\eta|Q)}\le C\eta(Q),\quad C=C(d,s),
\end{equation}
then the theorem follows. In fact, for $\eta\in\Sigma_s$ satisfying this condition, the norm of
Calder\'on-Zygmund operator on a space of non-homogeneous type can be estimated by a
constant depending only on $C$ in (\ref{f53}) and the constants in Calder\'on-Zygmund kernel, see
\cite{NTV97}, \cite{NTV03}, \cite{V}. In the spaces of homogeneous
type, it is the famous $T1$ theorem of David-Journ\'e (see \cite{DJ} for the Euclidean
setting and \cite{C} for homogeneous setting). Notice that a measure $\eta$
does not have any doubling property in general. So we cannot use the homogeneous $T1$ theorem, and we can use neither
\cite{DJ} nor \cite{C}. But the non-homogeneous $T1$ theorem \cite{NTV97}, \cite{NTV03} works fine. Thus,
$$
\|\mathfrak{R}_{\eta,\e}^s\|^2_{L^2(\eta)\to L^2(\eta)}\le C,\quad C=C(d,s),
$$
and we have
$$
\pmb|\mathfrak{R}_{\mu}^s\pmb|^2=2s\mathbf S\,\pmb|\mathfrak{R}_{\eta}^s\pmb|^2\le C\,\mathbf S,\quad C=C(d,s),
$$
that is \eqref{f411}. So, it is enough to prove \eqref{f53} or equivalently,

\begin{equation}\label{f54}
\|\mathfrak{R}_{\mu,\e}^s\chi_Q\|^2_{L^2(\mu|Q)}\le C\mathbf S\mu(Q),\quad C=C(d,s).
\end{equation}

We fix $\e>0$ and set
$$\aligned
\UU&=\{(x,y,z)\in Q^3:|y-x|>\e,\ |z-x|>\e\},\\
\UU_1&=\{(x,y,z)\in Q^3:|y-x|\ge|z-x|>\e\},\\
\UU_2&=\{(x,y,z)\in Q^3:\e<|y-x|\le|z-x|\},\\
\UU_{1,1}&=\{(x,y,z)\in Q^3:|y-x|\ge|z-x|>\e,\ |y-z|\ge|z-x|\},\\
\UU_{1,2}&=\{(x,y,z)\in Q^3:|y-x|\ge|z-x|>\e,\ |y-z|<|z-x|\}.
\endaligned$$
Then
$$\aligned
\int_{Q}|\mathfrak{R}_{\mu,\e}^s\chi_Q(x)|^2\,d\mu(x)&=\iiint_{\UU}
\frac{y-x}{|y-x|^{s+1}}\cdot\frac{z-x}{|z-x|^{s+1}}\,d\mu(z)\,d\mu(y)\,d\mu(x)\\
&\le\iiint_{\UU_1}+\iiint_{\UU_2}=:A+B.
\endaligned$$
It is enough to estimate $A$. We have
$$\aligned
|A|\le&\biggl|\iiint_{\UU_{1,1}}\frac{y-x}{|y-x|^{s+1}}\cdot\frac{z-x}{|z-x|^{s+1}}\,d\mu(z)\,d\mu(y)\,d\mu(x)\biggr|\\
+&\biggl|\iiint_{\UU_{1,2}}\frac{y-x}{|y-x|^{s+1}}\cdot\frac{z-x}{|z-x|^{s+1}}\,d\mu(z)\,d\mu(y)\,d\mu(x)\biggr|=:A_1+A_2.
\endaligned$$
We put the absolute value in $A_2$ inside the integral. Since $|z-x|>\frac12|y-x|$ in $A_2$, we get
\begin{equation}\label{f55}\aligned
A_2&\le2^s\int_Q\int_{|y-x|>0}\frac1{|y-x|^{2s}}\,\mu(B(x,|y-x|))\,d\mu(y)\,d\mu(x)\\
&=2^s\int_Q\int_0^\infty\frac1{r^{2s}}\,\mu(B(x,r))\,d\mu(B(x,r))\,d\mu(x)\\
&=2^s\int_Q\int_0^\infty\frac1{r^{2s}}\,d\biggl[\frac{\mu(B(x,r))^2}2\biggr]\,d\mu(x).
\endaligned\end{equation}
From (\ref{f52}) one can easily deduce that
\begin{equation}\label{f56}
\lim_{r\to0}\frac{\mu(B(x,r))}{r^{s}}=0
\end{equation}
(see, for example, \cite{MPV}, p.~219). Moreover,
$$
\frac{\mu(B(x,r))}{r^{s}}\le\frac{\mu(\R^d)}{r^{s}}\to0\text{ as }r\to\infty.
$$
Integrating by parts in the last integral of (\ref{f55}) we get
$$
A_2\le s2^s\int_Q\int_0^\infty\bigg[\frac{\mu(B(x,r))}{r^s}\bigg]^2\,\frac{dr}{r}\,d\mu(x)\le s2^s\mathbf S\mu(Q).
$$

Let us estimate $A_1$. By the symmetry of $\UU_{1,1}$ with respect to $z,x$ we have
$$
A_1=\frac12\biggl|\iiint_{\UU_{1,1}}\biggl(\frac{y-x}{|y-x|^{s+1}}\cdot\frac{z-x}{|z-x|^{s+1}}
+\frac{y-z}{|y-z|^{s+1}}\cdot\frac{x-z}{|x-z|^{s+1}}\biggr)\,d\mu(z)\,d\mu(y)\,d\mu(x)\biggr|.
$$
Lemma \ref{le51} yields
$$
A_1\le2^s\iiint_{\UU_{1,1}}\frac{1}{|y-x|^{s+1}}\cdot\frac{1}{|z-x|^{s-1}}\,d\mu(z)\,d\mu(y)\,d\mu(x).
$$
Clearly,
$$\aligned
A_1&\le2^s\int_Q\int_{|y-x|\ge\e}\frac1{|y-x|^{s+1}}\biggl[\int_\e^{|y-x|}\frac{d\mu(B(x,t))}{t^{s-1}}\biggr]\,d\mu(y)\,d\mu(x)\\
&\le2^s\int_Q\int_0^\infty\frac1{r^{s+1}}\biggl[\int_0^{r}\frac{d\mu(B(x,t))}{t^{s-1}}\biggr]\,d\mu(B(x,r))\,d\mu(x).
\endaligned$$
Set
$$
H_x(r):=\int_0^{r}\frac{d\mu(B(x,t))}{t^{s-1}}.
$$
Then the last expression can be written in the form
\begin{equation}\label{f57}
2^s\int_Q\int_0^\infty\frac1{r^{2}}H_x(r)\,dH_x(r)\,d\mu(x)=2^{s-1}\int_Q\int_0^\infty\frac{dH_x^2(r)}{r^2}\,d\mu(x).
\end{equation}
Obviously,
\begin{equation}\label{f58}
H_x(r)=\frac{\mu(B(x,r))}{r^{s-1}}+(s-1)\int_0^{r}\frac{\mu(B(x,t))}{t^{s}}\,dt,
\end{equation}
and
$$
\lim_{r\to\infty}\frac{H_x(r)}{r}=0,\quad \lim_{r\to0}\frac{H_x(r)}{r}=0
$$
(the last equality follows from (\ref{f56})). Thus,
\begin{equation}\label{f59}\aligned
\int_0^\infty\frac{dH_x^2(r)}{r^2}&=2\int_0^\infty\frac{H_x^2(r)}{r^3}\,dr
\overset{(\ref{f58})}\le4\int_0^\infty\bigg[\frac{\mu(B(x,r))}{r^s}\bigg]^2\,\frac{dr}{r}\\
&+4(s-1)^2\int_0^\infty\frac1{r^3}\biggl[\int_0^{r}\frac{\mu(B(x,t))}{t^{s}}\,dt\biggr]^2dr.
\endaligned\end{equation}
The first term in the right hand side of (\ref{f59}) is what we need. Let us estimate the second term.
By the Cauchy--Bunyakovskii--Schwarz inequality,
$$
\biggl[\int_0^{r}\frac{\mu(B(x,t))}{t^{s}}\,dt\biggr]^2\le
\int_0^{r}\bigg[\frac{\mu(B(x,t))}{t^s}\bigg]^2\,dt\cdot\int_0^rdt.
$$
Hence, applying integration by parts, we obtain
$$\aligned
\int_0^\infty\frac1{r^3}\biggl[\int_0^{r}\frac{\mu(B(x,t))}{t^{s}}\,dt\biggr]^2dr
&\le\int_0^\infty\biggl[\int_0^{r}\bigg[\frac{\mu(B(x,t))}{t^s}\bigg]^2\,dt\biggr]\frac{dr}{r^2}\\
&=\biggl(-\frac1r\int_0^{r}\bigg[\frac{\mu(B(x,t))}{t^s}\bigg]^2\,dt\biggr)\bigg|_0^\infty
+\int_0^{\infty}\bigg[\frac{\mu(B(x,r))}{r^s}\bigg]^2\,\frac{dr}r.
\endaligned$$
According to (\ref{f56}), the substitution of limits gives zero. Thus, (see (\ref{f59}))
$$
\int_0^\infty\frac{dH_x^2(r)}{r^2}<C(s)\int_0^{\infty}\bigg[\frac{\mu(B(x,r))}{r^s}\bigg]^2\,\frac{dr}r.
$$
Now (\ref{f57}) yields (\ref{f54}), and Theorem \ref{th46} is proved.

\vspace{.1in}

\noindent{\bf The second approach.} Notice that there was nothing specific in using cubes $Q$ in (\ref{f54}).
Verbatim the same proof gives more, namely, that for every measurable set $E$,

\begin{equation}\label{f510}
\|\mathfrak{R}_{\mu,\e}^s\chi_E\|^2_{L^2(\mu|E)}\le C\mathbf S\mu(E),\quad C=C(d,s).
\end{equation}
But operators $\mathfrak{R}_{\mu,\e}^s$ have one drawback: they are not operators with Calder\'on-Zygmund kernels.
So, instead of $\mathfrak{R}_{\mu,\e}^s$, we will use similar operators, but with Calder\'on-Zygmund kernels.
Let us introduce them. Let $\phi$ be a $C_0^{\infty}$ ``bell-like" function, $\phi=1$ on the unit ball and $\phi(x)=0$,
$|x|\ge2$. Put $\psi= 1-\phi$, and let $\psi_{\e}(\cdot):= \psi(\frac{\cdot}{\e})$. Set
$$\aligned
\widetilde{\mathfrak{R}}_{\mu,\e}^s f(x)&=\int \psi_{\e}(|x-y|)K^s(y-x)f(y)\,d\mu(y),\quad f\in L^2(\mu),\quad \e>0,\\
\widetilde{\mathfrak{R}}_{\mu,\ast}^s f(x)&=\sup_{\e>0}|\widetilde{\mathfrak{R}}_{\mu,\e}^s f(x)|.
\endaligned$$
We denote by $\widetilde{R}_{\nu,\e}^s$ the corresponding modified $s$-Riesz transform of a finite Borel measure $\nu$
(not necessarily positive):
$$
\widetilde{R}_{\nu,\e}^s(x)=\int \psi_{\e}(|x-y|)K^s(y-x)\,d\nu(y).
$$
It is easy to see that if $\eta\in \Sigma_s$ then
\begin{equation}
\label{f511}
|\mathfrak{R}_{\eta,\e}^s f(x) - \widetilde{\mathfrak{R}}_{\eta,\e}^s f(x)|
\le \frac{C}{\e^s}\int_{B(x,2\e)} |f|\,d\eta\le C(d,s)\widetilde{M} f(x)\,,
\end{equation}
where
\begin{equation}
\label{f512}
 \widetilde{M} f(x) := \sup_{r>0} \frac{1}{\eta(B(x,3r))}\int_{B(x,r)} |f|\,d\eta\,.
 \end{equation}
Maximal operator $\widetilde{M}$ is bounded in any $L^2(\eta)$ (see \cite{NTV}, Lemma~2.1). Hence, \eqref{f511} shows that
operators $\widetilde{\mathfrak{R}}_{\eta,\e}^s$ and $\mathfrak{R}_{\eta,\e}^s$ are bounded simultaneously, and that their norms
differ at most by $C$. Thus, \eqref{f510} yields the corresponding estimate for $\widetilde{\mathfrak{R}}_{\eta,\e}^s$:
\begin{equation}\label{f513}
\|\widetilde{\mathfrak{R}}_{\eta,\e}^s\chi_E\|^2_{L^2(\eta|E)}\le C\eta(E),\quad C=C(d,s).
\end{equation}
Suppose that we would have even more, namely that
\begin{equation}\label{f514}
\|\widetilde{\mathfrak{R}}_{\eta,\e}^s\chi_E\|^2_{L^2(\eta)}\le C\eta(E),\quad C=C(d,s).
\end{equation}
Using Theorem 5.1 of \cite{NTV} and \eqref{f514}, one can prove that operators
$$
\widetilde{\mathfrak{R}}_{\eta,\e}^s: L^1(\eta) \rightarrow L^{1,\infty}(\eta)
$$
are uniformly bounded in $\e$. Then Theorem 10.1 of \cite{NTV} proves that norms of operators $\widetilde{\mathfrak{R}}_{\eta,\e}^s$
are uniformly bounded in $L^2(\eta)$ by the same number (up to a constant).
This is what we need. The second proof of our theorem would be done.

\medskip

Unfortunately, we do not have \eqref{f514}. And \eqref{f513} alone  seems (at the first glance) to be too weak to carry
through the proofs of \cite{NTV}. It  is a pity because it would be nice to  prove Theorem \ref{th46} without referring
the reader to the proof of non-homogeneous $T1$ theorem from \cite{NTV97}.

However, there is a way to use  \eqref{f513} and to avoid the non-homogeneous $T1$ theorem. To do that, let us first
notice that \eqref{f513} can be ``strengthened" as follows: for any measurable sets $E$ and $F$,
\begin{equation}\label{f515}
\|\widetilde{\mathfrak{R}}_{\eta,\e}^s\chi_E\|^2_{L^2(\eta|F)}\le C(\eta(E)+ \eta(F)),\quad C=C(d,s).
\end{equation}
Indeed, using \eqref{f513} for various sets, we have
$$\aligned
\|\widetilde{\mathfrak{R}}_{\eta,\e}^s\chi_E\|^2_{L^2(\eta|F)}&=\|\widetilde{\mathfrak{R}}_{\eta,\e}^s\chi_{E\cup F}-
\widetilde{\mathfrak{R}}_{\eta,\e}^s\chi_{F\setminus E}\|^2_{L^2(\eta|F)}\\
&\le2\|\widetilde{\mathfrak{R}}_{\eta,\e}^s\chi_{E\cup F}\|^2_{L^2(\eta|E\cup F)}+
2\|\widetilde{\mathfrak{R}}_{\eta,\e}^s\chi_{F\setminus E}\|^2_{L^2(\eta|(F\setminus E)\cup(F\cap E))}\\
&\le2C\eta(E\cup F)+2C\eta(F\setminus E)+
2\|\widetilde{\mathfrak{R}}_{\eta,\e}^s\chi_F-\widetilde{\mathfrak{R}}_{\eta,\e}^s\chi_{F\cap E}\|^2_{L^2(\eta|F\cap E)}\\
&\le2C(\eta(E\cup F)+\eta(F\setminus E))
+4\|\widetilde{\mathfrak{R}}_{\eta,\e}^s\chi_F\|^2_{L^2(\eta|F)}+
4\|\widetilde{\mathfrak{R}}_{\eta,\e}^s\chi_{F\cap E}\|^2_{L^2(\eta|F\cap E)}\\
&\le2C(\eta(E\cup F)+\eta(F\setminus E)+2\eta(F)+2\eta(F\cap E))\le C'(\eta(E)+\eta(F)).
\endaligned$$

\begin{lemma}\label{2inf}
Let $\eta$ be a finite measure on some set $X$, and let function $f$ and constants $\tau,\ K\in(0,+\infty)$ be such
that for any measurable set $F$,
$$
\int_F |f|^2 d\eta \le \tau +K \eta(F)\,.
$$
Then $ f=f_1+f_2$, where $\|f_2\|_{\infty} \le 2K^{\frac12}$, and $\|f_1\|^2_{L^2(\eta)} \le \frac43\tau$.
\end{lemma}

\begin{proof}
Put $f_1:= f\chi_{\{|f| >2K^{\frac12}\}},\ f_2 :=  f\chi_{\{|f| \le 2K^{\frac12}\}}$. Then
$$
4K \eta(\{|f|>2K^{\frac12}\}) \le \int_{\{|f|>2K^{\frac12}\}} |f|^2\,d\eta \le \tau + K\eta(\{|f|>2K^{\frac12}\})\,,
$$
and therefore  $\eta(\{|f|>2K^{\frac12}\}) \le \frac{\tau}{3K}$. Hence,
$$
\int |f_1|^2\, d\eta  =  \int_{\{|f|>2K^{\frac12}\}} |f|^2\,d\eta \le \tau + K\eta(\{|f|>2K^{\frac12}\})\le \frac43\tau\,.
$$
The lemma is proved.
\end{proof}

Next we need the following modification of Guy David's lemma \cite{D} (or Lemma 4.1 from \cite{NTV}):

\begin{lemma}\label{GD}
For any operator $\widetilde{\mathfrak{R}}$ satisfying \eqref{f515}  with Calder\'on-Zygmund kernel, for any Borel set $E$, and for
any point $x\in\supp\eta$,
$$
\widetilde{\mathfrak{R}}_{*}\chi_E (x) \le C\,\widetilde{M} \widetilde{\mathfrak{R}}\chi_E (x) + A\,,\quad C=C(d)\,,
$$
where $A$ depends only on $d$ and  Calder\'on-Zygmund constants of the kernel.
\end{lemma}

\begin{proof}
We can try to repeat line by line the proof of Lemma 4.1 in \cite{NTV}, but we cannot use that $\widetilde{\mathfrak{R}}$ is $L^2(\eta)$--bounded. This is not given (and actually this is what we wish to prove), but instead
we just use \eqref{f515} to estimate the term on page 474 of \cite{NTV} as follows
$$
\bigg|\int\chi_{B(x,R)} \widetilde{\mathfrak{R}}\chi_{E\cap B(x,3R)}\,d\eta\bigg|
\le C(\eta(B(x, R) )^{\frac12}(\eta(B(x, R) + \eta(B(x,3R))^{\frac12}\,.
$$
This is enough to finish the proof exactly as in \cite{NTV}. In fact, the right hand side is at most $C\eta(B(x, R))$
because $3R$ is a doubling radius (being chosen in the proof  of Lemma 4.1 in \cite{NTV}).
\end{proof}

Notice a simple thing: the Calder\'on-Zygmund constants of the kernels of $\widetilde{\mathfrak{R}}_{\eta,\e}^s$  do not
depend on $\e$. This allows us to have uniform $C_0$ in  the following lemma.

\begin{lemma}\label{weakTlemma}
\begin{equation}\label{f516}
\|\widetilde{\mathfrak{R}}_{\eta,\e}^s\|_{L^1(\eta)\rightarrow L^{1,\infty}(\eta)} \le C_0\,,
\end{equation}
where $C_0$ depends only on $d$ and $s$ (and does not depend on $\e$).
\end{lemma}

\begin{proof}
It is sufficient to prove that
\begin{equation}\label{f517}
\|\widetilde{R}_{\nu,\e}^s\|_{L^{1,\infty}(\eta)} \le C\,\|\nu\|\,,
\end{equation}
where $\nu$ is a finite linear combination of unit point masses with positive coefficients. To show that, note first of
all that \eqref{f517} remains valid (with constant doubled) for $\nu$ with arbitrary real coefficients (just write $\nu$
as a difference of two measures with positive coefficients). Fix $f\in L^1(\eta)$, $\e>0$, and a ball $B\subset\R^d$.
Since the kernel of $\widetilde{\mathfrak{R}}_{\nu,\e}^s$ forms an equicontinuous family of functions
$k_x(y):=\psi_{\e}(|x-y|)K^s(y-x)$, we can approximate $f\,d\eta$ by a measure $\nu$ of the form $\nu=\sum_{j=1}^M
\a_j\d_{y_j}$, $M\in\mathbb N_+$, $\a_j\in\R$, in such a way that
$$
(\{|\widetilde{\mathfrak{R}}_{\eta,\e}^s f(x)|>t\}\cap B)\subset(\{|\widetilde{R}_{\nu,\e}^s(x)|>\tfrac12 t\}\cap B).
$$
With this end in view we cover a sufficiently large ball $B'\supset B$ by the net of small cubes $Q_j$, and take
$\a_j=\int_{Q_j}f\,d\eta$. Using \eqref{f517}, we have
$$
\eta(\{|\widetilde{\mathfrak{R}}_{\eta,\e}^s f(x)|>t\}\cap B)\le\eta(\{|\widetilde{R}_{\nu,\e}^s(x)|>\tfrac12 t\})
\le2C\,\|\nu\|\le2C\,\|f\|_{L^1(\eta)}\,.
$$
Since $B$ is arbitrarily large, we get \eqref{f516} with $C_0=2C$.

The proof of \eqref{f517} is a modification of the proof of Theorem 5.1 of \cite{NTV}. This latter can be repeated
line by line till we come to p.~477 of \cite{NTV}. There we need to estimate ($\widetilde{\mathfrak{R}}:=\widetilde{\mathfrak{R}}_{\eta,\e}^s$,
$\widetilde{\mathfrak{R}}_{*}:=\widetilde{\mathfrak{R}}_{\eta,\ast}^s$)
$$
 \widetilde{\mathfrak{R}}_{*}\chi_E \le C\, \widetilde{M} \widetilde{\mathfrak{R}} \chi_E + A\,,
 $$
 where constants $C,A$ depend only on $s,d$. Lemma 4.1 of \cite{NTV} is not applicable here, but we have the replacement,
 namely, Lemma \ref{GD} above.

Another hitch is that we are required to estimate the following expression:
$$
I:=\bigg|\int \chi_F \widetilde{M}  \widetilde{\mathfrak{R}} \chi_E\,d\eta\bigg|\,.
$$
Here
\begin{equation}\label{518}
\eta(F)= \eta(E)\,,
\end{equation}
 and one is tempted to use \eqref{f515} to obtain $I\le C\eta(E)$, which would be enough to finish
the proof verbatim as in \cite{NTV}, p.~477.
However, we cannot use here \eqref{f515}. In fact, \eqref{f515} does not have any $\widetilde{M}$ in it.

This is why here we need to be more subtle, and we need to use Lemma \ref{2inf}. Keep in mind \eqref{f515}
and apply Lemma \ref{2inf} to $f:= \widetilde{\mathfrak{R}} \chi_E,\ K= C,\ \tau = C\eta(E)$. Then we get
$$
\widetilde{\mathfrak{R}} \chi_E = f_1 + f_2,\ \ \|f_2\|_{\infty}\le C,\ \ \|f_1\|_2^2 \le C\eta(E)\,.
 $$
Hence,
$$
\|\widetilde{M}f_2\|_{\infty} \le C,\ \ \|\widetilde{M}f_1\|_{L^2(\eta)}^2 \le C'C\eta(E)\,,
$$
where $C':= \|\widetilde{M}\|_{L^2(\eta)\rightarrow L^2(\eta)}$. This constant is absolute.
As $I \le\int_F |\widetilde{M}f_1| \,d\eta +\int_F|\widetilde{M}f_2|\,d\eta$, it becomes obvious that
$$
I\le C''\eta(E)^{\frac12} \eta(F)^{\frac12} + C\eta(F)\,.
$$
We already mentioned \eqref{f517}. So,
$$
I \le C\,\eta(E)\,,
$$
and the proof of Lemma \ref{weakTlemma} is finished exactly as that of Lemma 5.1 of \cite{NTV}.
\end{proof}

We started from \eqref{f510} and we derived from it the uniform bound $C_0\mathbf S^{\frac12}$ for the norms
$\|\widetilde{\mathfrak{R}}_{\mu,\e}^s\|_{L^1(\mu)\rightarrow L^{1, \infty}(\mu)}$. Now we just use Theorem 10.1 from \cite{NTV},
which claims that for operator with Calder\'on-Zygmund kernel the boundedness of the latter norms implies its
boundedness in $L^2(\mu)$ by $C\,\mathbf S^{\frac12}$.

It is the time to recall that we have \eqref{f511}. Therefore, we obtained the uniform estimate for
$\|\mathfrak{R}_{\mu,\e}^s\|_{L^2(\mu)\rightarrow L^{2}(\mu)} $.
\end{proof}

From Theorem \ref{th46}, we derive a useful corollary for Cantor sets.
Let $\ell_0,\dots,\ell_n$ and $\l$ be such that
$$
0<\ell_{k+1}<\l \ell_k,\quad k=0,\dots,n-1,\quad 0<\l<1/2.
$$
For $N$ of the form $N=2^{nd}$ we consider $N$ Cantor cubes $Q_j^n,\ j=1,\dots,N$, built by the usual procedure
from the cube $Q_1^0$ with edge length $\ell_0$ by the corner construction, namely having $2^d$ corner cubes
$Q_j^1$ with edge length $\ell_1$, $2^{2d}$ cubes $Q_j^2$ with edge length $\ell_2$, et cetera.
Let $E_n=\bigcup_j Q_j^n$, and let $m$ be
the measure uniform on each $n$-cube and of mass $2^{-nd}$ on each $Q_j^n$. Set
$$
\theta_{s,k}=\theta_k=\frac{2^{-kd}}{\ell_k^s}.
$$

\begin{corollary}\label{cor52}
For the measure $m$ defined above,
\begin{equation}\label{f519}
\pmb|\mathfrak{R}_{m}^s\pmb|^2\le C\sum_{k=0}^n\theta_k^2,\quad C=C(d,s).
\end{equation}
\end{corollary}

\begin{proof}
Denote by $\rho$ the maximal density of $m$, that is $\rho=2^{-nd}/\ell_n^d$. For every $x\in E_n$ we have
$$
m(B(x,r))\le\left\{\begin{array}{ll}
\rho r^d,&0<r<\ell_n,\\
C2^{-kd},&\ell_k\le r\le \ell_{k-1},\ k=1,\dots,n,\\
1,&\ell_0\le r<\infty,
\end{array}\right.
$$
with the positive constant $C$ depending only on $d$. Hence, for every $x\in\R^d$,
$$\aligned
\int_0^{\infty}\bigg[\frac{m(B(x,r))}{r^s}\bigg]^2\,\frac{dr}r
&\le\int_0^{\ell_n}\biggl(\frac{2^{-nd}}{\ell_n^d}\biggr)^2r^{2d-2s-1}\,dr
+C\sum_{k=1}^n2^{-2kd}\int_{\ell_k}^{\ell_{k-1}}\frac{dr}{r^{2s+1}}+\int_{\ell_0}^{\infty}\frac{dr}{r^{2s+1}}\\
&\le\biggl(\frac{2^{-nd}}{\ell_n^d}\biggr)^2\frac{\ell_n^{2d-2s}}{2d-2s}+
C\sum_{k=1}^n\frac1{2s}\frac{2^{-2kd}}{\ell_k^{2s}}+\frac1{2s}\,\frac1{\ell_0^{2s}}\\
&\le C(d,s)\sum_{k=0}^n\biggl(\frac{2^{-kd}}{\ell_k^s}\biggr)^2.
\endaligned$$
Now (\ref{f519}) follows immediately from (\ref{f411}).
\end{proof}

It was proved in \cite{MT} that under the condition $\theta_{k+1}\le\theta_k$
\begin{equation}\label{f520}
C^{-1}\sum_{k=0}^n\theta_k^2\le\pmb|\mathfrak{R}_{m}^s\pmb|^2\le C\sum_{k=0}^n\theta_k^2,
\end{equation}
where $C$ depends on $\l$, $s$ and $d$. Thus, Theorem \ref{th46} is a generalization of the estimate from
above in (\ref{f520}).

Due to X.~Tolsa \cite{T1} we know now that the condition $\theta_{k+1}\le\theta_k$ is superfluous in the estimate from
below as well.

\section{Construction of the auxiliary measure}

We start with one property of Hausdorff contents.

\begin{lemma}\label{le61}
Let $h$ be a measuring function. For given $t_1>0$ we set

\begin{equation}\label{f61}
\overline{h}(t)=\left\{\begin{array}{ll}
t^dh(t_1)t_1^{-d},&0\le t<t_1,\\
h(t),&t\ge t_1.
\end{array}\right.
\end{equation}
If $\a\in(0,1)$, and the sets $F,G\in\R^d$ are such that
$$
B(x,\a t_1)\subset G\quad \text{for every } x\in F,
$$
then
$$
M_h(F)\le CM_{\overline{h}}(G),\quad C=C(\a,d).
$$
\end{lemma}

\begin{proof}
We cover $\R^d$ by the net of cubes $Q_l$ with edge length $\a t_1/\sqrt d$. For fixed $\e>0$ we take a
covering of $G$ by balls $B_j=B(x_j,r_j)$, such that
$$
M_{\overline{h}}(G)+\e>\sum_j \overline{h}(r_j).
$$
Let $Q_l$ be a cube from our net for which $Q_l\cap F\ne\varnothing$. Then $Q_l\subset G\subset\bigcup_j B_j$.
We replace each ball $B_j$ with $r_j\ge t_1$  by the ball $B'_j=2B_{j}$. We remark that
$h(r_j)\ge2^{-d}h(2r_j)$. If $Q_l$ intersects a ball
$B_j$ with $r_j\ge t_1$, then $Q_l\subset2B_j$. Suppose now that $Q_l$ intersects only balls $B_{j}$ with
$r_j<t_1$. These balls cover $Q_l$, and in turn can be covered and substituted by a ball of radius $3t_1$.
Since $\sum_{j:B_j\cap Q_l\ne\varnothing}r_j^d\ge c_d(\a t_1)^d$, we have
$$
\sum_{j:B_j\cap Q_l\ne\varnothing}\overline{h}(r_j)\ge
\frac{\overline{h}(3t_1)}{(3t_1)^d}\sum_{j:B_j\cap Q_l\ne\varnothing}r_j^d
\ge c(d,\a)\overline{h}(3t_1).
$$
We discard all (possibly) remaining balls $B_{j}$ with $r_j<t_1$. Thus, we obtain a new covering of $F$ by balls
$B'_{j}=B(x'_j,r'_j)$ with $r'_j\ge t_1$. Since each $B_{j}$ of $r_j<t_1$ intersects at most $C=C(d,\a)$ cubes
$Q_l$, we have
$$\aligned
M_{\overline{h}}(G)+\e&>\sum_j \overline{h}(r_j)>c(d,\a)\sum_j \overline{h}(r'_j)\\
&=c(d,\a)\sum_j h(r'_j)\ge c(d,\a)M_h(F).
\endaligned$$
\end{proof}

\begin{lemma}\label{le62}
Let $P>0$ be given, and let $\nu$ be a linear combination of $N$ Dirac point masses. There is a constant
$C_1$ depending only on $d$ and $s$, with the following property. If
\begin{equation}\label{f62}
\mathbf M:=M_h(\ZZ^\ast(\nu,P))>\frac{C_1\|\nu\|}P\max_{t_1\le t\le t_2}\frac{h(t)}{t^s},
\end{equation}
where $t_1=h^{-1}(0.1\mathbf M/N),\ t_2=h^{-1}(\mathbf M)$, then there exists a positive Borel measure
$\mu$ such that

1) $\supp\mu\subset\ZZ^\ast(\nu,0.8P)$;

2) $c\mathbf M\le\|\mu\|\le\mathbf M,\ c=c(d)$;

3) for every ball $B(x,r)\subset\R^d$, one has
\begin{equation}\label{f63}
\mu(B(x,r))\le\overline{h}(r),
\end{equation}
where $\overline{h}(r)$ is defined by (\ref{f61}).
\end{lemma}

For $d=2$ a similar assertion was proved in \cite{E2} (see Lemma~5.1 in \cite{E2}). The proof given here essentially
differs from the arguments in \cite{E2}.

\begin{proof}[Proof of Lemma \ref{le62}]
Let $|\nu|=\sum_{j=1}^N|\nu_j|\delta_{y_j}$ be the variation of $\nu$. First of all we exclude from $\R^d$ the set
with high density of $|\nu|$ and with comparatively ``small" $h$-content.

We say that a point $x\in\R^d$ is {\it normal} (with respect to $|\nu|$ and $h$) if the inequality
\begin{equation}\label{f64}
|\nu|(B(x,r))\le C_2^{-1}P\rho^{-1}h(r),\quad \rho:=\max_{t_1\le t\le t_2}\frac{h(t)}{t^s},
\end{equation}
holds for all $r\ge0$. Here $C_2<C_1$ is the constant depending only on $d$ and $s$, which will be specified later.
Let $G_1$ be the set of non-normal points $x\in\R^d$. For each $x\in G_1$ there exists $r=r(x)$ such that
\begin{equation}\label{f65}
h(r)<C_2P^{-1}\rho|\nu|(B(x,r)).
\end{equation}
We obtain a covering of $G_1$ by balls. Since (see (\ref{f62}), (\ref{f65}))
$$
h(t_2)=\mathbf M>\frac{C_1\|\nu\|}P\rho\ge\frac{C_1|\nu|(B(x,r))}{P}\rho>h(r),
$$
the radii of these balls are bounded by $t_2$. By Besicovitch's covering lemma (see for example \cite{Ma}, p.~30),
there is a subcovering $\{B'_k\}$, $B'_k=B(w'_k,r'_k)$, of multiplicity not exceeding $A_d$ (that is, every point
$x\in G$ is covered by at most $A_d$ balls $B'_k$). We set $\ZZ_1=\bigcup_k B'_k$. Then
\begin{equation*}
M_h(\ZZ_1)\le\sum_k h(r'_k)\overset{(\ref{f65})}<\frac{C_2\rho}{P}\sum_k|\nu|(B'_k)
\le \frac{A_dC_2\rho}{P}\|\nu\|\overset{(\ref{f62})}<0.3\mathbf M,
\end{equation*}
if $A_dC_2<0.3C_1$. We denote by $\{y_1,\dots,y_N\}$ the support of $\nu$. Set $\ZZ_2=\bigcup_{k=1}^NB(y_k,t_1)$. Then
$$
M_h(\ZZ_2)\le Nh(t_1)=0.1\mathbf M.
$$
We claim that
\begin{equation}\label{f66}
|\nu|(B(x,r))<C_2^{-1}Pr^s \quad\text{ for all } x\in\R^d\setminus(\ZZ_1\cup\ZZ_2)\text{ and } r>0.
\end{equation}
Indeed, $|\nu|(B(x,r))=0$ when $0\le r<t_1$. For $t_1\le r<t_2$, the inequality (\ref{f66}) follows
from (\ref{f64}). Assume that $r\ge t_2$. By (\ref{f62}),
$$
\mathbf M>\frac{C_1\|\nu\|}P\,\frac{h(t_2)}{t_2^s}
=\frac{C_1\|\nu\|}P\,\frac{\mathbf M}{t_2^s}.
$$
Hence,
$$
|\nu|(B(x,r))\le\|\nu\|<C_1^{-1}Pt_2^{s}\le C_1^{-1}Pr^s,
$$
and we get (\ref{f66}). We set
$$
F=\ZZ^\ast(\nu,P)\setminus(\ZZ_1\cup\ZZ_2).
$$
Clearly,
\begin{equation}\label{f67}
M_h(F)\ge\mathbf M-M_h(\ZZ_1)-M_h(\ZZ_2)>\mathbf M-0.3\mathbf M-0.1\mathbf M=0.6\mathbf M.
\end{equation}
We will prove that
\begin{equation}\label{f68}
B(x_0,0.4t_1)\subset\ZZ^\ast(\nu,0.8P)\quad\text{for every } x_0\in F.
\end{equation}
Fix $x_0\in F$, and let $x$ be such that $|x-x_0|<0.4t_1$.
Since $x_0\in\ZZ^\ast(\nu,P)$, there is $\e>0$, for which $|R_{\nu,\e}^s(x_0)|>P$. We may assume that $\e>0.9\,t_1$.
(Indeed, $|\nu|(B(x_0,r))=0$ as $0\le r<t_1$. Hence, all integrals $R_{\nu,\e}^s(x_0)$ are the same for $0<\e<t_1$.)
Obviously,
\begin{multline}\label{f69}
|R_{\nu,\e}^s(x_0)-R_{\nu,\e-|x-x_0|}^s(x)|\\
\le\int_{|y-x_0|>\e}\biggl|\frac{y-x_0}{|y-x_0|^{s+1}}-\frac{y-x}{|y-x|^{s+1}}\biggr|\,d|\nu|(y)
+\int_{\{|y-x|>\e-|x-x_0|\}\cap\{|y-x_0|\le\e\}}\frac{1}{|y-x|^{s}}\,d|\nu|(y).
\end{multline}
Since $|y-x_0|>\e>0.9\,t_1$ implies $|y-x|>|y-x_0|-0.4t_1>0.5|y-x_0|$, we have
\begin{equation*}
\biggl|\frac{y-x_0}{|y-x_0|^{s+1}}-\frac{y-x}{|y-x|^{s+1}}\biggr|
\le C(s)\frac{|x-x_0|}{|y-x_0|^{s+1}}<\frac{C(s)\e}{|y-x_0|^{s+1}}.
\end{equation*}
Hence, the first integral on the right hand side of (\ref{f69}) does not exceed
\begin{multline*}
C(s)\e\int_{|y-x_0|>\e}\frac{d|\nu|(y)}{|y-x_0|^{s+1}}=
C(s)\e\int_{\e}^\infty\frac{d|\nu|(B(x_0,t))}{t^{s+1}}\\
\le C(s)\e\int_{\e}^\infty\frac{(s+1)|\nu|(B(x_0,t))}{t^{s+2}}\,dt
\le\frac{C(s)\e(s+1)P}{C_2}\int_{\e}^\infty\frac{dt}{t^2}<0.1P,
\end{multline*}
if $C_1$ and $C_2$ are big enough (we integrated by parts and used (\ref{f66})). Using (\ref{f66}) again we see
that the second integral on the right hand side of (\ref{f69}) is bounded by
$$
\frac{|\nu|(\overline{B}(x_0,\e))}{(\e-|x-x_0|)^{s}}<\frac{P\e^s}{C_2(0.5\e)^s}<0.1P,
$$
if $C_2>10\cdot2^s$. Thus,
$$
|R_{\nu,\e}^s(x_0)-R_{\nu,\e-|x-x_0|}^s(x)|<0.2P\quad\text{for all } x\in B(x_0,0.4t_1).
$$
Since $|R_{\nu,\e}^s(x_0)|>P$ and $x_0$ is any point in $F$, we get (\ref{f68}).

We consider the compact set
$$
E=\overline{\bigcup_{x\in F}B(x,0.2t_1)}.
$$
By (\ref{f68}), $E\subset\ZZ^\ast(\nu,0.8P)$. By Frostman's theorem (see for example \cite{Ca}, p.~7)
there is a measure $\mu$ supported by $E$ such that
$$
\|\mu\|\ge a_d M_{\overline{h}}(E)\quad\text{and}\quad \mu(B(x,r))\le\overline{h}(r)\text{ for every }B(x,r)\subset\R^d.
$$
Lemma \ref{le61} yields
$$
\|\mu\|\ge a_d M_{\overline{h}}(E)\ge c(d)M_h(F)\overset{(\ref{f67})}>c'(d)\mathbf M.
$$

It remains to consider the inequality $\|\mu\|\le\mathbf M$. If $\|\mu\|>\mathbf M$, we multiply $\mu$ by the constant
$\mathbf M/\|\mu\|<1$, fulfilling in this way all the requirements of Lemma \ref{le62}.
\end{proof}

\section{Proof of the first part of Theorem \ref{th41}}

As above, we set $\mathbf M=M_h(\ZZ^\ast(\nu,P))$. It is enough to prove the inequality
\begin{equation}\label{f71}
\mathbf M\le C(s,d)\frac{\|\nu\|}{P}
\left[\int_{t_1}^{t_2}\bigg(\frac{h(t)}{t^s}\bigg)^2\frac{dt}{t}\right]^{1/2},
\end{equation}
where $t_1=h^{-1}(0.1\,\mathbf M/N),\ t_2=h^{-1}(\mathbf M)$. Then Lemma \ref{le32} will imply the desired assertion.

We remark that
\begin{equation}\label{f72}
\max_{t_1\le t\le t_2}\frac{h(t)}{t^s}=:\frac{h(t_0)}{t_0^s}<C_3
\left[\int_{t_1}^{t_2}\bigg(\frac{h(t)}{t^s}\bigg)^2\frac{dt}{t}\right]^{1/2},\quad C_3=C_3(d,s).
\end{equation}
Indeed, we have
$$
0.05>\frac{0.1}N=\frac{h(t_1)}{h(t_2)}=\frac{h(t_1)}{t_1^d}\,\frac{t_2^d}{h(t_2)}\,
\bigg(\frac{t_1}{t_2}\bigg)^d\ge\bigg(\frac{t_1}{t_2}\bigg)^d.
$$
Hence, $t_1\le(0.05)^{1/d}t_2$. There is the interval $[\a,\beta]\subset[t_1,t_2]$, containing $t_0$ and such that
$\a=c(d)\beta$, $0<c(d)<1$. Since $h(t)>c'(d)h(t_0)$ as $t\in[\a,\beta]$, we get
\begin{equation*}
\frac{h(t_0)}{t_0^s}<C(d,s)
\left[\int_{\a}^{\beta}\bigg(\frac{h(t)}{t^s}\bigg)^2\frac{dt}{t}\right]^{1/2},
\end{equation*}
that implies (\ref{f72}). So, if
$$
\mathbf M\le\frac{C_1\|\nu\|}{P}\max_{t_1\le t\le t_2}\frac{h(t)}{t^s},
$$
then (\ref{f71}) holds, and the first part of Theorem \ref{th41} is proved. Thus, we may assume that (\ref{f62}) is
fulfilled. Let $\mu$ be the measure in Lemma \ref{le62}. Relations (\ref{f63}) and $\|\mu\|\le\mathbf M=h(t_2)$
imply the estimate
$$
\frac{\mu(B(x,r))}{r^s}\le\max_{t_1\le t\le t_2}\frac{h(t)}{t^s} \quad\text{for every ball }\ B(x,r)\subset\R^d.
$$
Set
$$
a_h:=C_3\left[\int_{t_1}^{t_2}\bigg(\frac{h(t)}{t^s}\bigg)^2\frac{dt}{t}\right]^{1/2},\quad \eta:=a_h^{-1}\mu.
$$
Then $\eta\in\Sigma_s$ (see (\ref{f72}) and (\ref{f49})). Moreover, by (\ref{f61}) for every $x\in\R^d$, we have
$$
\int_0^\infty\bigg[\frac{\mu(B(x,r))}{r^s}\bigg]^2\,\frac{dr}{r}\le\frac{h^2(t_1)}{t_1^{2d}}\int_0^{t_1}t^{2d-2s-1}\,dt+
\int_{t_1}^{t_2}\bigg[\frac{h(t)}{t^s}\bigg]^2\,\frac{dt}t+\int_{t_2}^{\infty}\frac{\|\mu\|^2}{t^{2s+1}}\,dt\le C(d,s)a_h^2
$$
(we remind the reader that $\|\mu\|\le\mathbf M=h(t_2)$). Theorem \ref{th46} yields
\begin{equation}\label{f73}
\pmb|\mathfrak{R}_{\eta}^s\pmb|^2=a_h^{-2}\pmb|\mathfrak{R}_{\mu}^s\pmb|^2\le C,\quad C=C(d,s).
\end{equation}
We apply Theorem \ref{th45} with $t=0.8P$ and $\eta$ instead of $\mu$. By (\ref{f73}), the constant $C$ in
(\ref{f410}) depends only on $d$ and $s$. Since
$\supp\eta\subset\ZZ^\ast(\nu,0.8P)$, (\ref{f410}) and the properties 1), 2) in Lemma \ref{le62} imply
\begin{equation*}
\frac{C\|\nu\|}{0.8P}>\eta\big(\ZZ^\ast(\nu,0.8P)\big)=\|\eta\|=a_h^{-1}\|\mu\|\ge c(d)a_h^{-1}\mathbf M,
\end{equation*}
that is equivalent to (\ref{f71}).\hfill$\square$

\section{Proof of the second part of Theorem \ref{th41}}

Without loss of generality we may assume that $N=2^{nd}$. Fix $\mathbf M>0$, and let
$$
\ell_j=h^{-1}(2^{-dj}\mathbf M),\ j=0,\dots,n-1,\quad \ell_n=\frac15h^{-1}(2^{-dn}\mathbf M).
$$
Since $t^{-d}h(t)$ is non-increasing, $\ell_{j+1}\le\frac12\ell_j$. Choose the set $J\subset\{0,\dots,n\}$
inductively as follows: $0\in J$; if $j\in J$, then the least $k>j$ such that $\ell_k\le\frac152^{j-k}\ell_j$,
also belongs to $J$. If $j_0,j_1,\dots,j_m$ are the elements of $J$ listed in the increasing order, then
$j_0=0$, $j_m=n$, and for every $k=0,\dots,m-1$ we have
$$
\ell_{j_k}\ge2\ell_{j_k+1}\ge4\ell_{j_k+2}\ge\dots\ge2^{j_{k+1}-j_k-1}\ell_{j_{k+1}-1}
\ge\frac15\ell_{j_k}\ge 2^{j_{k+1}-j_k}\ell_{j_{k+1}}\,.
$$

Construct the random set $E$ recursively as follows: $E_m$ is just the cube with edge length $\ell_{j_m}$
centered at the origin. Suppose that $E_{k+1}$ is already defined as a random set. Let $Q$ be the cube with
edge length $\frac15\ell_{j_k}$ centered at the origin. Partition it into $D_k:=2^{d(j_{k+1}-j_k-1)}$
equal subcubes. Let $x_1,\dots,x_{D_k}$ be the centers of those subcubes. Take $D_k$ independent copies of
$E_{k+1}$ and define $F_k$ to be the union of those copies shifted by $x_1,\dots,x_{D_k}$. Now take $2^d$
independent copies of $F_k$ and shift them by
$$
\frac{\ell_{j_k}}5(\e+v_\e),
$$
where $\e$ runs over all $2^d$ vectors in $\R^d$ whose coordinates are $\pm1$, and $v_\e$ are independent
random vectors uniformly distributed over the cube with edge length $\frac1{10}$ centered at the origin and
also independent of all $F_k$. The resulting random set is $E_k$.

It is easy to show by induction that each $E_k$ is contained in the cube with edge length $\ell_{j_k}$
centered at the origin and, for $k<n$, the $2^d$ shifted copies of $F_k$ whose union is $E_k$ are separated
by $\frac1{10}\ell_{j_k}$. Note also that $E=E_0$ consists of $2^{nd}$ randomly located
cubes with edge length $\ell_n$, which we will call the base cubes.


\begin{figure}[h]
\includegraphics{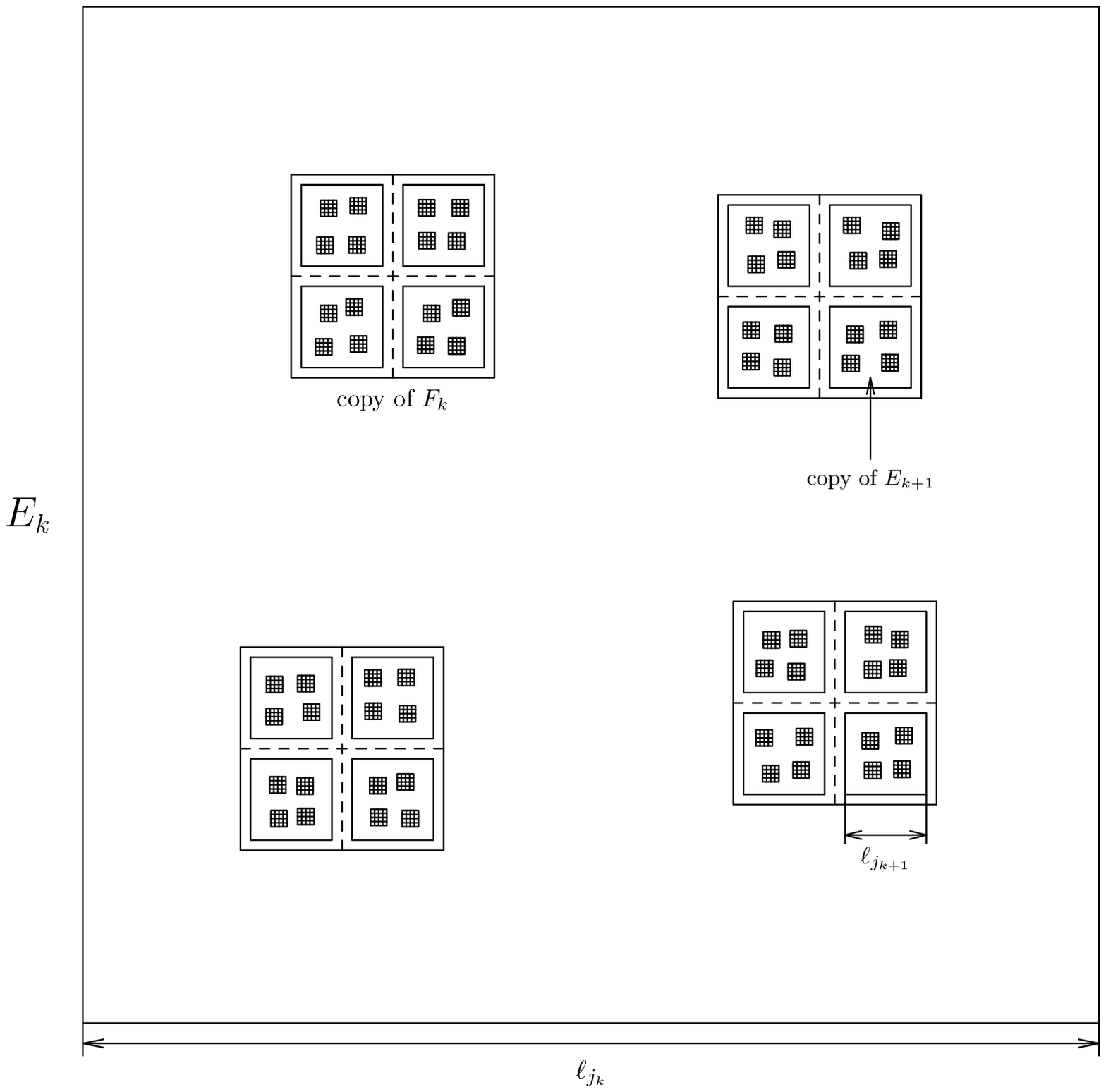}
\end{figure}

Denote by $\widetilde E$ the sure set constructed exactly in the same way as $E$, but with $v_\e=0$.

Define the random measures $\mu$ and $\nu$ supported by $E$ as follows. For each base cube $Q$, we put
$$
\mu(Q)=2^{-nd}\,,\,\,\nu(Q)=2^{-nd}\eta.
$$
The measure $\mu$ will be proportional to Lebesgue measure on each base cube $Q$, and the measure $\nu$
will be a multiple of the Dirac point mass located at the center of $Q$.

\begin{lemma}\label{le81}
For every ball $B_r\subset\R^d$ of radius $r>0$, we have
$$
\mathbf M\mu(B_r)\le C h(r)\ \text{ with some }\ C=C(d,s)>0.
$$
\end{lemma}

\begin{proof}
Note first of all that the base cubes are disjoint. Thus, the density of $\mu$ with respect to the Lebesque
measure is not greater than
$$
\frac{2^{-nd}}{\ell_n^d}=\frac{5^d\mathbf M^{-1}h(5\ell_n)}{(5\ell_n)^d}.
$$
Hence,
$$
\mu(B_r)\le C\mathbf M^{-1}r^d\,\frac{h(r)}{r^d}=C\mathbf M^{-1}h(r),
$$
if $r<5\ell_n$.

Suppose now that $\ell_0\ge r\ge5\ell_n$. Then $r\in(\ell_{j_{k+1}},\ell_{j_k}]$ for some $k$. Note that $B_r$
can intersect only a bounded number of random copies of $E_k$ constituting $E$ because each such copy lies in its
own cube with edge length $\ell_{j_k}$.

Now $E_k$ consists of $2^d$ blocks $F_k$. In each block the $\mu$-measure of every cube of edge length
$$
\frac{\ell_{j_k}}5\,2^{j_k+1-j_{k+1}}\ge2\ell_{j_{k+1}}
$$
is $2^{-j_{k+1}d}$. Since $B_r$ can intersect at most
$$
\bigg[\frac{5r}{\ell_{j_k}}\,2^{j_{k+1}-j_k}+1\bigg]^d\le
C_d\bigg[\bigg(\frac{5r}{\ell_{j_k}}\,2^{j_{k+1}-j_k}\bigg)^d+1\bigg]
$$
such cubes, we conclude that
$$\aligned
\mu(B_r)&\le C r^d\,\frac{2^{-j_kd}}{\ell_{j_k}^d}+C 2^{-j_{k+1}d}\le
C \mathbf M^{-1}r^d\,\frac{h(\ell_{j_k})}{\ell_{j_k}^d}+C \mathbf M^{-1}h(5\ell_{j_{k+1}})\\
&\le C \mathbf M^{-1}r^d\,\frac{h(r)}{r^d}+C'\mathbf M^{-1}h(\ell_{j_{k+1}})
<C'' \mathbf M^{-1}h(r).
\endaligned$$
Finally, if $r\ge\ell_0$, we have
$$
\mu(B_r)\le\mu(E)=1=\mathbf M^{-1}h(\ell_0)\le\mathbf M^{-1}h(r).
$$
\end{proof}

\begin{corollary}\label{co82}
For every Borel set $G\subset\R^d$, we have
$$
M_h(G)\ge C^{-1}\mathbf M\mu(G),\ \ C=C(d,s)>0.
$$
\end{corollary}

\begin{lemma}\label{le83}
Let $\theta_k>0,\ k=0,\dots,m-1$. Let $\xi_k$ be $\R^d$-valued independent random variables satisfying
$$
|\xi_k|\le C\theta_k,\quad \sum_k\Var\xi_k\ge c\sum_k\theta_k^2.
$$
Then there exists $\d=\d(C,c,d)>0$ such that
$$
\mathscr{P}\bigg\{\bigg|\sum_{k=0}^{m-1}\xi_k+a\bigg|\ge\d\bigg(\sum_{k=0}^{m-1}\theta_k^2\bigg)^{1/2}\bigg\}\ge\d
\ \text{ for all }\ a\in\R^d.
$$
\end{lemma}

\begin{proof}
Denote
$$
\s=\sum_{k=0}^{m-1}\xi_k+a,\quad \zeta_k=\xi_k-\mathscr{E}\xi_k.
$$
It is enough to consider the case $d=1$. Indeed, since $\int_{|e|=1}\Var\langle\xi_k,e\rangle\,dm(e)=\frac1d\Var\xi_k$,
we can find a unit vector $e$ such that $\sum_k\Var\langle\xi_k,e\rangle\ge\frac cd\sum_k\theta_k^2$. But, obviously,
$\langle\xi_k,e\rangle\le C\theta_k$ and $|\s|\ge|\sum_k\langle\xi_k,e\rangle+\langle a,e\rangle|$.

Take $\l>0$ and consider
$$
|\mathscr{E}e^{i\l\s}|=\bigg|e^{i\l a}\prod_{k=0}^{m-1}\mathscr{E}e^{i\l\xi_k}\bigg|\\
=\prod_{k=0}^{m-1}|\mathscr{E}e^{i\l\xi_k}|=\prod_{k=0}^{m-1}|\mathscr{E}e^{i\l\zeta_k}|.
$$
Note now that
$$
e^{i\l\zeta_k}=1+i\l\zeta_k-\frac{\l^2}2\zeta_k^2+O(\l^3\zeta_k^3).
$$
Hence,
$$
|\mathscr{E}e^{i\l\zeta_k}|=\bigg|1-\frac{\l^2}2\Var\xi_k+C\l^3\theta_k^3\bigg|\le
\exp\bigg(-\frac{\l^2}2\Var\xi_k+C\l^3\theta_k^3\bigg),
$$
and
$$\aligned
\prod_{k=0}^{m-1}|\mathscr{E}e^{i\l\zeta_k}|&\le
\exp\left(-c\l^2\sum\theta_k^2+C\l^3\sum\theta_k^3\right)\\
&\le\exp\bigg(-c'\left[\l\left(\sum\theta_k^2\right)^{1/2}\right]^2+
C'\left[\l\left(\sum\theta_k^2\right)^{1/2}\right]^3\bigg).
\endaligned$$
Now choose
$$
\l=\frac{c'}{2C'}\left(\sum\theta_k^2\right)^{-1/2}.
$$
Then
$$
|\mathscr{E}e^{i\l\s}|\le\exp\bigg(-\frac{(c')^3}{8(C')^2}\bigg).
$$
On the other hand, for every $\d>0$, one has
$$\aligned
|\mathscr{E}e^{i\l\s}-1|&\le\l\d\left(\sum\theta_k^2\right)^{1/2}+
2\mathscr{P}\bigg\{|\s|>\d\left(\sum\theta_k^2\right)^{1/2}\bigg\}\\
&\le\frac{c'}{2C'}\d+2\mathscr{P}\bigg\{|\s|>\d\left(\sum\theta_k^2\right)^{1/2}\bigg\}.
\endaligned$$
Hence,
$$
\mathscr{P}\bigg\{|\s|>\d\left(\sum\theta_k^2\right)^{1/2}\bigg\}\ge
\frac12\bigg[1-\exp\bigg(-\frac{(c')^3}{8(C')^2}\bigg)-\frac{c'}{2C'}\d\bigg]>\d,
$$
if $\d$ is chosen small enough.
\end{proof}

\begin{lemma}\label{le84}
Let $\nu$ be a positive measure supported by the cube $Q$ with edge length $\ell$ centered at the origin.
Let $x\in[\frac54\ell,\frac{11}4\ell]^d$. Let $v$ be the random vector uniformly distributed over the
cube with edge length $\frac12\ell$ centered at the origin. Then the random variable $\xi=R_{\nu}^s(x+v)$
satisfies
$$
|\xi|\le C\theta,\quad \Var\xi\ge c\theta^2,\ \text{ where }\theta=\frac{\|\nu\|}{\ell^s}.
$$
\end{lemma}

\begin{proof}
Since for every $j=1,\dots,d$, and every $y\in Q$, we have
$$
(x+v-y)_j\ge\frac{5\ell}4-\frac{\ell}2-\frac{\ell}4\ge\frac{\ell}2,
$$
we conclude that
$$
|R_{\nu}^s(x+v)|\le\frac{2^s}{\ell^s}\|\nu\|=2^s\theta.
$$
Now we have to consider two cases.

{\bf Case 1:} $d=1$.

Then
$$\aligned
\D&=|\mathscr{E}\{R_{\nu}^s(x+v)|v<0\}-\mathscr{E}\{R_{\nu}^s(x+v)|v>0\}|\\
&=\frac4{\ell}\int_{[-\ell/4,0]}\bigg[R_{\nu}^s(x+v)-R_{\nu}^s\left(x+v+\frac{\ell}4\right)\bigg]\,dv\\
&=\frac4{\ell}\int_{[-\ell/4,0]}dv\int_Q\bigg[\frac1{(x+v-y)^s}-\frac1{(x+v+\frac14\ell-y)^s}\bigg]\,d\nu(y).
\endaligned$$
Note now that
$$
\frac{\partial}{\partial t}\,\frac1{t^s}=-\frac s{t^{s+1}}\le-\frac s{(\frac72\ell)^{s+1}},\
\text{ whenever }\ t\in\bigg[x+v-y,\,x+v+\frac{\ell}4-y\bigg]\subset\bigg[\frac{\ell}2,\,\frac{7\ell}2\bigg].
$$
Thus, the difference of the conditional expectations is at least
$$
s\,\bigg(\frac27\bigg)^{s+1}\,\frac14\,\frac{\|\nu\|}{\ell^s}=c\theta.
$$
Since $\Var\xi\ge\D^2/4$, we are done in this case.

{\bf Case 2:} $d\ge2$.

Let now
$$
\D=|\mathscr{E}\{[R_{\nu}^s(x+v)]_1|v_2<0\}-\mathscr{E}\{[R_{\nu}^s(x+v)]_1|v_2>0\}|.
$$
Since for $t\in[\frac12\ell,\,\frac72\ell]^d$, we have
$$
\frac{\partial}{\partial t_2}\,\frac{t_1}{|t|^{s+1}}=-(s+1)\frac {t_1t_2}{|t|^{s+3}}\le-\frac c{\ell^{s+1}}\,.
$$
Denote $\xi^{(1)} := (\xi,e_1)$. Then
the arguments analogous to that for the one-dimensional case allows us to conclude that
$$
\Var\xi\ge\Var\xi^{(1)}\ge\frac{\D^2}4\ge c\theta^2.
$$
\end{proof}

Fix $x\in\widetilde{E}$ and consider the random variable $\s=R_\nu^s(\Phi(x))$, where
$\Phi:\widetilde{E}\mapsto E$ is the (random) canonical measure preserving mapping between
$\widetilde{E}$ and $E$ defined above. Namely, $\Phi$ shifts each base cube of $\widetilde{E}$
onto the corresponding perturbed base cube of $E$.

Now fix the random shifts responsible for the positioning of the copies of the blocks $F_k$
$(k=0,\dots,m-1)$, that contain $\Phi(x)$.

Suppose now that some copy of the block $F_k$ contains $\Phi(x)$. Then fix all shifts responsible
for the positioning of the base cubes within the copies $F_k$ not containing $\Phi(x)$. That will
leave free $2^d-1$ shifts of the copies of the block $F_k$ at each of the levels $k=0,\dots,m-1$.
Out of these shifts, fix all except the one responsible for the positioning of the copy of $F_k$
``opposite'' in each coordinate to the copy containing $\Phi(x)$ in the copy of $E_k$ containing $\Phi(x)$.

Now notice that $R_\nu^s(\Phi(x))$ is the sum of a fixed part, that we shall call $a$, and $m$
independent random parts $\xi_0,\dots,\xi_{m-1}$, where $\xi_k$ is the potential of the part of the
measure $\mu$ supported by the copy of $F_k$ in the copy of $E_k$ containing $\Phi(x)$ that is
``opposite'' to the copy of $F_k$ containing $\Phi(x)$.

Applying Lemma \ref{le84} to $\xi_k$ with $\ell=\frac15\ell_{j_k}$ (in Lemma \ref{le84} we moved the
point and fixed the measure while here we are moving the measure and fixing the point, but for
convolution kernels it is the same; also the obvious change of directions of the coordinate axes should
be performed before we find ourselves in the conditions of Lemma \ref{le84}), we conclude that
$$
|\xi_k|\le C\,\frac{\eta2^{-j_kd}}{\ell_{j_k}^s},\ \text{ and }\
\Var\xi_k\ge C\bigg(\frac{\eta2^{-j_kd}}{\ell_{j_k}^s}\bigg)^2,\ \ k=0,1,\dots,m-1.
$$
Hence, Lemma \ref{le83} is applicable and we get
$$
\mathscr{P}\bigg\{|R_\nu^s(\Phi(x))|\ge\d\bigg(\sum_{k=0}^{m-1}\theta_k^2\bigg)^{1/2}\bigg\}\ge\d
\ \text{ with }\ \theta_k=\frac{\eta2^{-j_kd}}{\ell_{j_k}^s}.
$$
This is a conditional probability,  of course, conditioned on freezing a lot of shifts. But, since $\d$
and $\theta_k$ do not depend on the values of the frozen shifts, we get the same inequality for the full
probability.

The immediate conclusion is that there exists a set of shifts such that
$$
\mu(G)\ge\d, \ \text{ where }\
G=\bigg\{y\in E:|R_\nu^s(y)|\ge\d\bigg(\sum_{k=0}^{m-1}\theta_k^2\bigg)^{1/2}\bigg\},
$$
and, by Corollary \ref{co82}, $M_h(G)\ge c\mathbf M$.

Now note that
$$\aligned
\int_{C/N}^1\bigg[\frac{t}{h^{-1}(\mathbf{M}t)^s}\bigg]^2\frac{dt}{t}
&\le\sum_{j=0}^{n-2}\int_{2^{-(j+1)d}}^{2^{-jd}}\bigg[\frac{t}{h^{-1}(\mathbf{M}t)^s}\bigg]^2\frac{dt}{t}\\
&\le C\sum_{j=0}^{n-2}\bigg[\frac{2^{-jd}}{h^{-1}(\mathbf{M}\,2^{-(j+1)d})^s}\bigg]^2
=C\sum_{j=1}^{n-1}\bigg[\frac{2^{-jd}}{\ell_j^s}\bigg]^2\\
&\le C\sum_{k=0}^{m-1}\sum_{j_k\le j<j_{k+1}}\bigg[\frac{2^{-jd}}{\ell_j^s}\bigg]^2
\le C\sum_{k=0}^{m-1}\bigg[\frac{2^{-j_kd}}{\ell_{j_k}^s}\bigg]^2
\sum_{j\ge j_k}2^{-2(j-j_k)(d-s)}\\
&\le C\sum_{k=0}^{m-1}\bigg[\frac{2^{-j_kd}}{\ell_{j_k}^s}\bigg]^2=C\eta^{-2}\sum_{k=0}^{m-1}\theta_k^2.
\endaligned$$
Thus, we constructed a linear combination of $N$ Dirac point masses $\nu$ such that $\|\nu\|=\eta$ and
$$
M_h\bigg(\ZZ\bigg(\nu,c\eta\bigg\{\int_{C/N}^1\bigg[\frac{t}{h^{-1}(\mathbf{M}t)^s}\bigg]^2\frac{dt}{t}
\bigg\}^{1/2}\bigg)\bigg)\ge c\mathbf{M}.
$$
Now, given $P>0$, put
$$
\mathbf M=\mathfrak{M}_h\bigg(\frac{c\eta}P,\frac NC\bigg).
$$
Then
$$
c\eta\bigg\{\int_{C/N}^1\bigg[\frac{t}{h^{-1}(\mathbf{M}t)^s}\bigg]^2\frac{dt}{t}
\bigg\}^{1/2}=P.
$$
It remains to note that
$$
\mathfrak{M}_h\bigg(\frac{c\eta}P,\frac NC\bigg)\ge c'\mathfrak{M}_h\bigg(\frac{\eta}P,N\bigg).
$$
This proves the theorem for $N\ge2C$.

For $2\le N<2C$, just put all $N$ masses to one point. Then
$$
M_h(\ZZ(\nu,P))=h\bigg(\bigg(\frac{\eta}P\bigg)^{1/s}\bigg).
$$
Note now that
$$
\k^2\int_{1/2}^1\bigg[\frac{t}{h^{-1}(2h(\k^{1/s})t)^s}\bigg]^2\frac{dt}{t}
\le\k^2\int_{1/2}^1\bigg[\frac{t}{\k}\bigg]^2\frac{dt}{t}<1,
$$
so $h(\k^{1/s})\ge\frac12\mathfrak{M}_h(\k,2)$. Thus,
$$
h\bigg(\bigg(\frac{\eta}P\bigg)^{1/s}\bigg)\ge\frac12\mathfrak{M}_h\bigg(\frac{\eta}P,2\bigg)
\ge c\,\mathfrak{M}_h\bigg(\frac{\eta}P,2C\bigg)\ge c\,\mathfrak{M}_h\bigg(\frac{\eta}P,N\bigg).
$$

\section{Proof of Theorem \ref{th43}}

At the beginning we consider a measure $\nu$ with compact support.
Without loss of generality we assume that
\begin{equation*}
\mathbf M:=M_h(\ZZ^\ast(\nu,P))>\frac{C_1\|\nu\|}P\rho_0,\ \text{ where }\
\rho_0:=\max_{0<t\le t_2}\frac{h(t)}{t^s},\ \ h(t_2)=\mathbf M.
\end{equation*}
Otherwise the same arguments as in Section 7 (with sufficiently big $N$) yield (\ref{f46}). In the same way
as in the proof of Lemma \ref{le62} we define the set $\ZZ_1$ (we use the notation in Lemma \ref{le62}),
taking $\rho_0$ instead of $\rho$. Repeating the arguments in the proof of Lemma \ref{le62} and choosing
$C_2>10$, we see that
\begin{gather}\label{f91}
M_h(\ZZ_1)<0.3\mathbf M,\\
|\nu|(B(x,r))<0.1Pr^s \quad\text{ for any } x\in\R^d\setminus\ZZ_1\text{ and } r>0.
\end{gather}
For fixed $\d>0$ we set
\begin{equation*}
\ZZ_\d^\ast(\nu,P)=\{x\in\R^d:\sup_{\e>\d}|R_{\nu,\e}^s(x)|>P\}.
\end{equation*}
There exists a measure $\nu'$ consisting of finitely many point charges such that
\begin{equation}\label{f93}
\|\nu'\|=\|\nu\|,\quad |R_{\nu,\e}^s(x)-R_{\nu',\e}^s(x)|<0.2P
\quad\text{ for any } x\in\R^d\setminus\ZZ_1\text{ and } \e\ge\d.
\end{equation}
For the construction of $\nu'$ we cover $\R^d$ by a sufficiently fine net of cubes and put at the center of
each cube a charge equal to the $\nu$-measure of this cube (the parts of the measure concentrated on sides
of adjacent cubes can be ascribed to any of them). The difference of the integrals over cubes lying in the
domain $\{y\in\R^d:|y-x|>\e\ge\d\}$ can be made arbitrarily small, while the difference of the integrals over
those pieces of cubes intersecting the sphere $|y-x|=\e$ which lie in this domain, can be made not exceeding
$0.15P$ (we use here (9.2)). For $\d=1/n$ we denote this measure $\nu'$ by $\nu_n$. By (\ref{f93}),
\begin{equation*}
\ZZ_{1/n}^\ast(\nu,P)\setminus\ZZ_1\subset\ZZ_{1/n}^\ast(\nu_n,0.8P)\setminus\ZZ_1\subset\ZZ^\ast(\nu_n,0.8P),
\ \ n\in\mathbb{Z}_+.
\end{equation*}
Obviously,
$$
\ZZ_{1/n}^\ast(\nu,P)\setminus\ZZ_1\nearrow\ZZ^\ast(\nu,P)\setminus\ZZ_1.
$$
Then
$$
c(d)M_h(\ZZ^\ast(\nu,P)\setminus\ZZ_1)\le\lim_{n\to\infty}M_h(\ZZ_{1/n}^\ast(\nu,P)\setminus\ZZ_1).
$$
This inequality follows from arguments given by Carleson in \cite{Ca}, p.~9--11 (see also \cite{E2}, Lemma~7.1).
By (\ref{f91}), $0.7\mathbf M\le M_h(\ZZ^\ast(\nu,P)\setminus\ZZ_1)$. Applying Theorem \ref{th41} for $\nu=\nu_n$
and Lemma \ref{le31}, for sufficiently big $n$ we get
$$
\mathbf M\le CM_h(\ZZ^\ast(\nu_n,0.8P))\le C'\,\Mh\bigg(\frac{\|\nu\|}{P},\infty\bigg),
$$
that is (\ref{f46}).

Suppose that $\supp\nu$ is not bounded.

Fix $R>0$ and take $R_1>R$. Let $\nu_0$ be a measure such that $\nu=\nu_0$ in $B(0,R_1)$, and
$\nu_0=0$ outside of $B(0,R_1)$. For sufficiently big $R_1$,
$$
\{\ZZ^\ast(\nu,P)\cap B(0,R)\}\subset\{\ZZ^\ast(\nu_0,0.9P)\cap B(0,R)\}\subset\ZZ^\ast(\nu_0,0.9P).
$$
Hence,
\begin{equation}\label{f94}
M_h(\ZZ^\ast(\nu,P)\cap B(0,R))\le C(d,s)\,\Mh\bigg(\frac{\|\nu_0\|}{0.9P},\infty\bigg)\le
C'(d,s)\,\Mh\bigg(\frac{\|\nu\|}{P},\infty\bigg)\ \text{ for all }R>0
\end{equation}
(it is clear that Lemma \ref{le31} is correct for $N=\infty$ as well). We will prove that this inequality
implies (\ref{f46}).

Choose $R>0$ such that $h(\frac12R)-1$ is greater then the right hand side of (\ref{f94}), and set
$$
G_1=\ZZ^\ast(\nu,P)\cap B(0,R),\quad G_k=\ZZ^\ast(\nu,P)\cap\{y:(k-1)R\le|y|<kR\}.
$$
Fix $\e\in(0,1)$ and $K\in\mathbb N$, and consider a covering of $\ZZ^\ast(\nu,P)\cap B(0,KR)$ by balls
$B_j=B(x_j,r_j)$ such that
$$
\sum_jh(r_j)<M_h(\ZZ^\ast(\nu,P)\cap B(0,KR))+\e\le C'(d,s)\,\Mh\bigg(\frac{\|\nu\|}{P},\infty\bigg)+\e.
$$
Every ball $B_j$ intersects at most two sets $G_k$, $k=1,\dots,K$. Hence,
$$
2C'(d,s)\,\Mh\bigg(\frac{\|\nu\|}{P},\infty\bigg)+2\e>2\sum_jh(r_j)\ge
\sum_{k=1}^K\sum_{j:B_j\cap G_k\ne\varnothing}h(r_j)\ge\sum_{k=1}^K M_h(G_k),
$$
and we get the estimate
$$
M_h(\ZZ^\ast(\nu,P))\le\sum_{k=1}^\infty M_h(G_k)\le2C'(d,s)\,\Mh\bigg(\frac{\|\nu\|}{P},\infty\bigg).
$$
{}\hfill $\square$

\section{The case $s\ge d$: proof of Theorem \ref{th44}}

{\bf Case $s>d$.} Let $K(t),\ t\in(0,+\infty)$, be a non-increasing continuous function such that
$K(t)\to+\infty$ as $t\to+0$. By \cite{E2}, Theorem~1.1,
$$
M_h\left(\left\{x\in\R^d: \int K(|x-y|)\,d|\nu|(y)>P\right\}\right)<Ch(r_0),\quad C=C(d),
$$
where $r_0$ is a solution of the equation
$$
h(r_0)=\frac{\|\nu\|}{P}\bigg(h(r'_0)K(r'_0)+\int_{r'_0}^{r_0}K(t)\,dh(t)\bigg)
\quad\text{with }\ h(r'_0)=\frac{h(r_0)}{N}.
$$
We apply this result for $K(t)=t^{-s},\ s>d$. Integrating by parts, we see that the right hand side does
not exceed
$$
\frac{\|\nu\|}{P}\bigg(\frac{h(r_0)}{r_0^s}+s\int_{r'_0}^{r_0}\frac{h(t)}{t^{s+1}}\,dt\bigg).
$$
Since $h(t)\le h(r'_0)(r'_0)^{-d}t^d$ for $t\ge r'_0$, we have
$$
h(r_0)\le\frac{\|\nu\|}{P}\bigg(\frac{h(r_0)}{r_0^s}+\frac{sh(r'_0)}{(r'_0)^d}\,\frac{(r'_0)^{d-s}}{s-d}\bigg)
\le C(s,d)\frac{\|\nu\|}{P}\,\frac{h(r'_0)}{(r'_0)^s}.
$$
Hence,
$$
r'_0\le\bigg(C(s,d)\frac{\|\nu\|}{NP}\bigg)^{1/s}.
$$
Therefore,
$$\aligned
M_h(\ZZ^\ast(\nu,P))&\le M_h(\XX(|\nu|,P))\le Ch(r_0)=CNh(r'_0)\\
&<C'Nh\bigg(\bigg(\frac{\|\nu\|}{NP}\bigg)^{1/s}\bigg),\quad C'=C'(s,d)
\endaligned$$
(we used the inequality $h(2t)\le2^dh(t)$).

\medskip

{\bf Case $s=d$.} Without loss of generality we assume that
$$
h(t_2)=\mathbf M:=M_h(\ZZ^\ast(\nu,P))>C_1Nh\bigg(\bigg(\frac{\|\nu\|}{NP}\bigg)^{1/d}\bigg)
$$
with sufficiently big $C_1$. Set
$$
t_1=\bigg(\frac{\|\nu\|}{NP}\bigg)^{1/d},\quad \rho=\frac{h(t_1)}{t_1}.
$$
As in the proof of Lemma 6.2, we construct the sets $\ZZ_1$, $\ZZ_2$ and $F$ such that
\begin{gather*}
F=\ZZ^\ast(\nu,P)\setminus(\ZZ_1\cup\ZZ_2),\quad M_h(F)\ge0.6\mathbf M,\\
B(x,0.4t_1)\subset\ZZ^\ast(\nu,0.8P)\ \text{ for every }\ x\in F.
\end{gather*}
Let $\mu$ be the $d$-dimensional Lebesgue measure, and let
$$
G=\bigcup_{x\in F}B(x,0.4t_1).
$$
Obviously, $G\subset\ZZ^\ast(\nu,0.8P)$. By the classical Calder\'on-Zygmund result,
$$
\frac{C\|\nu\|}{P}>\mu(\ZZ^\ast(\nu,0.8P))\ge\mu(G),\quad C=C(d).
$$
The set $G$ can be covered by balls $B_j$ with the same radii $t_1$ in such a way that
$$
\mu(G)\ge c\sum_j t_1^d,\quad c=c(d).
$$
Obviously,
$$\aligned
\mu(G)&\ge c\sum_j t_1^d=\frac{c\,t_1^d}{h(t_1)}\sum_j h(t_1)\ge\frac{c\,t_1^d}{h(t_1)}M_h(G)\\
&\ge\frac{c\,t_1^d}{h(t_1)}M_h(F)\ge\frac{c't_1^d}{h(t_1)}\mathbf M,\quad c'=c'(d).
\endaligned$$
Thus,
$$
\mathbf M\le\frac{C\|\nu\|}{P}\,\frac{h(t_1)}{t_1^d}=\frac{C\|\nu\|}{P}\,\frac{NP}{\|\nu\|}h(t_1)
=CNh(t_1).
$$
Theorem \ref{th44} is proved.\hfill$\square$

\section{Hausdorff content and capacity}

The main object of this section is the capacity $\g_{s,+}(E)$ of a compact set $E\subset\R^d$ defined by the equality
$$
\g_{s,+}(E):=\sup\{\|\mu\|:\mu\in M_+(E),\ \|R_{\mu}^s(x)\|_\infty\le1\},
$$
where $M_+(E)$ is the class of positive Radon measures supported by $E$.

\vspace{.09in}

{\bf Remark.} In \cite{V}, p.~46, the capacity $\g_{s,+}(E)$ is defined in the following way:
$$
\g_{s,+}(E):=\sup\{\|\mu\|:\mu\in \Sigma_s,\ \supp\mu\subset E,\ \|R_{\mu}^s(x)\|_\infty\le1\}.
$$
It is shown in \cite{MPV}, p.~217, that if $\|R_{\mu}^s(x)\|_\infty\le1$, then
$$
\mu(B(x,r))\le Cr^s,\quad x\in \R^d,\ r>0
$$
for every measure $\mu\in M_+(E)$. Arguments in this part of the proof of Lemma~4.1 in \cite{MPV} valid not only for $0<s<1$, but for $0<s<d$ as well.
(We note that the reference [P], Lemma~11 in \cite{MPV} should be replaced by [P], Lemma~3.1.) For $s=d-1$, this fact is also noted in \cite{V}, p.~46.
Therefore, these two definitions of $\g_{s,+}$ are equivalent.

\vspace{.09in}

This capacity is connected with various problems in analysis.

For $d=2,\ s=1$,
$$
\g_{1,+}(E)\asymp\g(E),
$$
where $\g(E)$ is analytic capacity (see \cite{To} and \cite{V}). Here $A\asymp B$ means that $C^{-1}A\le B\le CA$ with $C$ depending (possibly)
only on $d$ and $s$.

For $s=d-1,\ d\ge2$,
\begin{equation}\label{f111}
\g_{s,+}(E)\asymp\kappa(E),
\end{equation}
where $\kappa(E)$ is the Lipschitz harmonic capacity
$$
\kappa(E):=\sup\{|\langle\D f,1\rangle|: f\in\text{Lip}_{\text{loc}}^1(\R^d),\
\supp(\D f)\subset E,\ \|\nabla f\|_\infty\le1,\ \nabla f(\infty)=0\},
$$
introduced by Paramonov \cite{Pa} in connection with problems of approximation by harmonic functions. Here (as usual)
$\langle T,\f\rangle$ means the action of a distribution $T$ with compact support on a smooth test function. It was noticed
in \cite{Pa} that $\kappa(E)\le2\pi\g(E)$ for $d=2$. The
relation (\ref{f111}) was proved in \cite{To} for $d=2$ and in \cite{V} for $d>2$ (see \cite{V}, Theorem~2.1 and Lemma~5.15). The null-sets for the capacity $\kappa$ are the same as the
removable sets for Lipschitz harmonic functions, see \cite{Pa}, \cite{MP}. In these papers Mattila and Paramonov established important geometrical
properties of the capacity $\kappa$.

Moreover, $\g_{s,+}(E)$ is related to the Riesz capacity $C_{\a,p}$ in non-linear potential theory.
We discuss this relation near the end of the paper.

Main results of this section concern the connections between Hausdorff content and the capacity $\g_{s,+}$.
We need the following important characterization of $\g_{s,+}$ obtained in \cite{V}, Chapter~5:
\begin{equation}\label{f113}
\g_{s,+}(E)\asymp\g_{op}(E):=\sup\{\|\mu\|:\mu\in \Sigma_s,\ \supp\mu\subset E,\
\pmb|\mathfrak{R}_{\mu}^s\pmb|\le1\},\quad 0<s<d.
\end{equation}

\begin{proof}[Proof of Theorem \ref{th47}]
Choose $\mu\in M_+(E)$. We may assume that $\int_{\R^d}W^\mu(x)\,d\mu(x)<\infty$. Set
$$
G:=\bigg\{x\in\R^d:W^\mu(x)>\frac2{\|\mu\|}\int_{\R^d}W^\mu(x)\,d\mu(x)\bigg\}.
$$
It is easy to see that $G$ is open and
$$
\mu(G)\le\frac12\|\mu\|.
$$
Let
$$
\mu^\ast=\mu|(\R^d\setminus G), \quad \mathbf S=\sup_{x\in\supp\mu^\ast}W^{\mu^\ast}(x).
$$
We claim that
\begin{equation}\label{f117}
W^{\mu^\ast}(x)\le2^{2s+1}\mathbf S\quad \text{for all }x\in\R^d.
\end{equation}
It is enough to consider $x$ with $\d:=\text{dist}(x,\supp\mu^\ast)>0$. Let $x'$ be such that $x'\in\supp\mu^\ast$ and $|x-x'|=\d$. Then
$$\aligned
W^{\mu^\ast}(x)&=\int_\d^\infty\bigg[\frac{\mu^\ast(B(x,r))}{r^s}\bigg]^2\,\frac{dr}r\le
\int_\d^\infty\bigg[\frac{\mu^\ast(B(x',r+\d))}{r^s}\bigg]^2\,\frac{dr}r\\
&=\int_{2\d}^\infty\bigg[\frac{\mu^\ast(B(x',t))}{(t-\d)^s}\bigg]^2\,\frac{dt}{t-\d}
<2^{2s+1}\int_{2\d}^\infty\bigg[\frac{\mu^\ast(B(x',t))}{t^s}\bigg]^2\,\frac{dt}t\le2^{2s+1}\mathbf S,
\endaligned$$
and we get (\ref{f117}).

Let $\eta=(2^{2s+2}s\mathbf S)^{-1/2}\mu^\ast$. Since for each ball $B(x,r)$
$$
2^{2s+1}\mathbf S\ge\int_0^\infty\bigg[\frac{\mu^\ast(B(x,t))}{t^s}\bigg]^2\,\frac{dt}t\ge
\int_r^\infty\bigg[\frac{\mu^\ast(B(x,t))}{t^s}\bigg]^2\,\frac{dt}t\ge
\frac{[\mu^\ast(B(x,r))]^{2}}{2sr^{2s}},
$$
we see that $\eta\in\Sigma_s$. Moreover, (\ref{f411}) implies
$$
\pmb|\mathfrak{R}_{\eta}^s\pmb|^2\le C(2^{2s+2}s\mathbf S)^{-1}\mathbf S=C'.
$$
Relations (\ref{f113}) and $\|\mu^\ast\|\ge\frac12\|\mu\|$ yield
$$
\g_{s,+}(E)\ge C\eta(E)\ge C'\|\mu\|\mathbf S^{-1/2}.
$$
Since
$$
\mathbf S\le\sup_{x\in\supp\mu^\ast}W^{\mu}(x)\le\frac2{\|\mu\|}\int_{\R^d}W^\mu(x)\,d\mu(x),
$$
we have
$$
\g_{s,+}(E)\ge C\|\mu\|^{3/2}\bigg[\int_{\R^d}W^\mu(x)\,d\mu(x)\bigg]^{-1/2},
$$
that implies (\ref{f411'}).
\end{proof}

\begin{theorem}\label{th111}
Under assumption (\ref{f45}), for each compact set $E\subset\R^d$,
\begin{equation}\label{f114}
\g_{s,+}(E)\ge CM_h(E)\bigg[\int_0^{t_2}\bigg(\frac{h(t)}{t^{s}}\bigg)^2\frac{dt}t\bigg]^{-1/2},\quad 0<s<d,
\end{equation}
where $C$ depends only on $d$, $s$, and $t_2$ is defined by the equality $h(t_2)=M_h(E)$.
\end{theorem}

\begin{proof}
By Frostman's theorem (see \cite{Ca}, p.~7) there is a positive measure $\mu$ such that

$\supp\mu\subset E,$

$\mu(B(x,r))\le h(r)$ for each ball $B(x,r)\subset\R^d,$

$\mu(E)\ge CM_h(E)$ with $C$ depending only on $d$.\\
Without loss of generality we can assume that $\|\mu\|\le M_h(E)$ (otherwise we divide $\mu$ by the constant
$\|\mu\|/M_h(E)>1$). Then
$$
W^\mu(x)\le\int_0^{t_2}\bigg[\frac{h(t)}{t^{s}}\bigg]^2\frac{dt}t+
\int_{t_2}^\infty\bigg[\frac{h(t_2)}{t^{s}}\bigg]^2\frac{dt}t\le
C\int_0^{t_2}\bigg[\frac{h(t)}{t^{s}}\bigg]^2\frac{dt}t,
$$
since by the doubling property of $h$,
$$
\bigg[\frac{h(t_2)}{t_2^{s}}\bigg]^2\le C\int_0^{t_2}\bigg[\frac{h(t)}{t^{s}}\bigg]^2\frac{dt}t.
$$
Inequality (\ref{f114}) follows directly from (\ref{f411'}).
\end{proof}

For $h(t)=t^\beta$ easy calculations give the following result.

\begin{corollary}\label{cor112}
For each compact set $E\subset\R^d$,
$$
\g_{s,+}(E)\ge C(\beta-s)^{1/2}[M_h(E)]^{s/\beta},\quad\text{where}\quad 0<s<d,\quad h(t)=t^\beta,\quad\beta>s,
$$
and $C$ depends only on $d$ and $s$.
\end{corollary}

The next statement can be viewed as a counterpart of the classical Frostman's theorem on connections between capacities generated by potentials with
positive kernels and Hausdorff measure $\L_h(E)$ (see, for example, \cite{Ca}, Section IV, Theorem~1).

\begin{corollary}\label{cor113}
For each compact set $E\subset\R^d$, the condition $\g_{s,+}(E)>0$ implies $\L_h(E)>0$ for $h(t)=t^s$.

On the other hand, if $\L_h(E)>0$ for a measuring function $h$ satisfying (\ref{f45}), then $\g_{s,+}(E)>0$.
\end{corollary}

\begin{proof}
The first part of Corollary \ref{cor113} is a direct consequence of the following result by Prat \cite{Pr}, p.~946: for $0<s<d$
$$
C_\e[M_{t^{s+\e}}(E)]^{s/(s+\e)}\le\g_s(E)\le CM_{t^s}(E)
$$
(we need the second inequality). Indeed, by definition $\g_{s,+}(E)\le\g_s(E)$, and $M_h(E)$, $\L_h(E)$ vanish simultaneously. (We remark that for $0<s<1$, Prat \cite{Pr} has obtained the following essentially
stronger result: if $\g_s(E)>0$ then $\L_h(E)=\infty$.)

The second part is an immediate consequence of (\ref{f114}).
\end{proof}

Obviously, there is a gap between the assumptions about $h$ in the first and the second parts of Corollary \ref{cor113}. We claim that this gap cannot be
reduced, that is, both parts are sharp. Concerning the first part it means that if $\liminf_{t\to0}h(t)t^{-s}=0$, then there is a compact set $E$ for
which $\g_{s,+}(E)>0$ but $\L_h(E)=0$. This assertion follows from the more general and strong result
\cite{Ca}, p.~34, Theorem~4: for any positive decreasing
kernel $K(r)$ and any measuring function $h(r)$ such that
$$
\liminf_{r\to0}h(r)\overline{K}(r)=0,
$$
there is a Cantor type set $E$ with $C_K(E)>0$ and $\L_h(E)=0$. Here
\begin{gather*}
\overline{K}(r)=\frac1{r^d}\int_0^r K(t)t^{d-1}\,dt,\\
C_K(E):=\sup\bigg\{\|\mu\|:\mu\in M_+(E),\ \int_{\R^d}K(|x-y|)\,d\mu(y)\le1\text{ on }E\bigg\}.
\end{gather*}
For $K(r)=r^{-s}$ we have  $\overline{K}(r)=\frac1{d-s}r^{-s}$. By the maximum principle, $\g_{s,+}(E)\ge C\cdot C_K(E)$,
and we get the needed assertion.

The second part of Corollary \ref{cor113} is also precise: if the integral in (\ref{f45}) is divergent, then there exists a compact set $E$ for which $\L_h(E)>0$ but $\g_{s,+}(E)=0$. The industrious reader can obtain
this claim from Section~8. The reader who does not care about conditions of regularity of $h$,
can derive this statement under the additional condition $\frac{h(t)}{t^s}\nearrow$, using the estimate for
the capacity $\g_s$ of Cantor sets given at the end of \cite{MT}.

The results of this section mentioned above generalize the corresponding results in \cite{E2}, Section~12.

\medskip

\noindent{\bf Remark.} In Carleson's book \cite{Ca} (see also \cite{E2}, Sections 1,\,2 and the references therein)
another condition similar to \eqref{f45} plays an important part, namely
\begin{equation}\label{f47}
\int_0\frac{h(t)}{t^{s}}\,\frac{dt}t<\infty.
\end{equation}

It is interesting that the difference between these two conditions is explained by the fact that we are dealing here with
capacities with {\it signed} kernels, and \eqref{f47} is pertinent to the classical capacities with the {\it positive}
kernel $K(|x|)=|x|^{-s}$.

\medskip

In conclusion we consider the relation between $\g_{s,+}(E)$ and the Riesz capacity $C_{\a,p}$ in non-linear potential
theory. One of a number of equivalent definitions is the following equality (see \cite{AH}, p.~34, Theorem~2.5.1):
$$
C_{\a,p}(E)=\sup_{\mu\in M_+(E)}\bigg(\frac{\mu(E)}{\|I_\a\ast\mu\|_{p'}}\bigg)^p,\quad I_\a(x)=\frac{A_{d,\a}}{|x|^{d-\a}},\quad \frac1{p'}+\frac1p=1,
$$
where $1<p<\infty$, $0<\a p\le d$, $\|\cdot\|_{p'}$ is the $L^{p'}$-norm with respect to the Lebesque
measure in $\R^d$, and $A_{d,\a}$ is the certain constant depending on $d$ and $\a$. It was proved in \cite{MPV} that
\begin{equation}\label{f112}
\g_{s,+}(E)\asymp\g_s(E)\asymp C_{\frac23(d-s),\frac32}(E),\quad 0<s<1,
\end{equation}
where
$$
\g_s(E):=\sup|\langle T,1\rangle|,
$$
and the supremum is taken over all distributions $T$ supported by $E$ such that $T\ast\dfrac{x}{|x|^{s+1}}$ is a function in $L^\infty(\R^d)$ with
$\left\|T\ast\dfrac{x}{|x|^{s+1}}\right\|_\infty\le1$. Theorem \ref{th47} implies that one of the inequalities between
$\g_{s,+}(E)$ and $C_{\frac23(d-s),\frac32}(E)$ holds for wider range of $s$.

\begin{corollary}\label{cor48}
For $0<s<d$, one has
\begin{equation}\label{f412}
\g_{s,+}(E)\ge c\cdot C_{\frac23(d-s),\frac32}(E).
\end{equation}
\end{corollary}

\begin{proof}
We may assume that $C_{\frac23(d-s),\frac32}(E)>0$. As in \cite{MPV}, our proof is based on the following Wolff's equality
\cite{AH}, p.~110, Theorem~4.5.4: for any $\mu\in M_+(\R^d)$ and $1<p<\infty,\ 0<\a p\le d$,
\begin{equation}\label{f115}
\int_{\R^d}W_{\a,p}^\mu(x)\,d\mu(x)\asymp\|I_\a\ast\mu\|_{p'}^{p'},\quad W_{\a,p}^\mu(x):=
\int_0^\infty\bigg[\frac{\mu(B(x,r))}{r^{d-\a p}}\bigg]^{p'-1}\frac{dr}{r}.
\end{equation}
Take $\a=\frac23(d-s)$, $p=\frac32$. Then $p'=3$, $d-\a p=s$, and
$$
W_{\a,p}^\mu(x)=\int_0^\infty\bigg[\frac{\mu(B(x,r))}{r^s}\bigg]^2\,\frac{dr}r=:W^\mu(x).
$$
Choose $\mu\in M_+(E)$ for which
\begin{equation}\label{f116}
C_{\a,p}(E)<2\|\mu\|^p\|I_\a\ast\mu\|_{p'}^{-p}.
\end{equation}
By  (\ref{f411'}) and (\ref{f115}) we have
$$
\g_{s,+}(E)\ge c\,\|\mu\|^{3/2}\|I_\a\ast\mu\|_{p'}^{-p'/2}\overset{(\ref{f116})}>\frac{c}2\,C_{\frac23(d-s),\frac32}(E),
$$
and we get (\ref{f412}).
\end{proof}

For integer $s\in(0,d)$ the opposite inequality $\g_{s,+}(E)\le C\cdot C_{\frac23(d-s),\frac32}(E)$ is false. In fact, for a smooth $s$-dimensional manifold $E$ in $\R^d$ we have $\g_{s,+}(E)> 0$ by the obvious reason that natural Lebesgue measure on it gives bounded Riesz transform operator (this is from the classical Calder\'on-Zygmund theory). On the other hand, it has been noticed (for example  in \cite{MPV}) that any measure $\mu$ with finite Wolff's energy should have $\mu(B(x,r) ) =o(r^s)$ for $\mu$ a. e. $x$. On a smooth $s$-dimensional manifold it can be only  zero measure, so $C_{\frac23(d-s),\frac32}(E)=0$.
The question about validity of the inequality $\g_{s,+}(E)\le C\cdot C_{\frac23(d-s),\frac32}(E)$ for all non-integer $s\in(0,d)$ remains open.
We believe that this is the case.

\end{document}